\documentclass[12pt,oneside,english,reqno]{amsart}
\usepackage[T1]{fontenc}
\usepackage[latin9]{inputenc}
\usepackage[a4paper]{geometry}
\usepackage[active]{srcltx}
\usepackage{babel}
\usepackage{varioref}
\usepackage{textcomp}
\usepackage{amsthm}
\usepackage{amsbsy}
\usepackage{amstext}
\usepackage{amssymb}
\usepackage{esint}
\usepackage[unicode=true,pdfusetitle,
 bookmarks=true,bookmarksnumbered=false,bookmarksopen=false,
 breaklinks=false,pdfborder={0 0 1},backref=false,colorlinks=false]
 {hyperref}
\usepackage{breakurl}
\usepackage{cite}

\makeatletter
\numberwithin{equation}{section}
\numberwithin{figure}{section}
\theoremstyle{plain}
\newtheorem{thm}{\protect\theoremname}
  \theoremstyle{remark}
  \newtheorem{rem}{\protect\remarkname}
  \theoremstyle{definition}
  \newtheorem{defn}{\protect\definitionname}
  \theoremstyle{plain}
  \newtheorem{prop}{\protect\propositionname}

\makeatother

  \providecommand{\definitionname}{Definition}
  \providecommand{\propositionname}{Proposition}
  \providecommand{\remarkname}{Remark}
\providecommand{\theoremname}{Theorem}

\begin{document}

\title[Hyperbolic quasiperiodic solutions of U-monotone systems ]{Hyperbolic quasiperiodic solutions of U-monotone systems on Riemannian
manifolds}

\author{Igor Parasyuk}

\address{Faculty of Mechanics and Mathematics, Taras Shevchenko National University
of Kyiv, 64/13, Volodymyrska Street, City of Kyiv, Ukraine, 01601 }

\email{pio@univ.kiev.ua}

\begin{abstract}
We consider a second order non-autonomous system which can be interpreted
as the Newtonian equation of motion on a Riemannian manifold under
the action of time-quasiperiodic force field. The problem is to find
conditions which ensures: (a) the existence of a solution taking values
in a given bounded domain of configuration space and possessing a
bounded derivative; (b) the hyperbolicity of such a solution; (c)
the uniqueness and, as a consequence, the quasiperiodicity of such
a solution. Our approach exploits  ideas of Wa\.zewski topological
principle. The required conditions are formulated in terms of an auxiliary
convex function $U$. We use this function to establish the Landau
type inequality for the derivative of solution, as well as to introduce
the notion of\emph{ $U$}-monotonicity for the system. The \emph{U}-monotonicity
property of the system implies the uniqueness and the quasiperiodicity
of its bounded solution. We also find the bounds for magnitude of
perturbations which do not destroy the quasiperiodic solution.

The results obtained are applied to study the motion of a charged
particle on a unite sphere under the action of time-quasiperiodic
electric and magnetic fields.
\end{abstract}


\maketitle

\section{Introduction}

Let $(\mathcal{M},\left\langle \cdot,\cdot\right\rangle )$ be a smooth
complete connected $m$-dimensional Riemannian manifold with the metric
tensor $\mathfrak{g}=\left\langle \cdot,\cdot\right\rangle $, and
let $\nabla$ be the Levi-Civita connection with respect to $\mathfrak{g}$.
For a given smooth mapping $x(\cdot):I\mapsto\mathcal{M}$ of an interval
$I\subset\mathbb{R}$, denote by $\nabla_{\dot{x}}\dot{x}(t)$ the
covariant derivative of tangent vector field $\dot{x}(\cdot):I\mapsto T\mathcal{M}$
along $x(\cdot)$ at the point $t\in I$. Here $T\mathcal{M}=\bigsqcup_{x\in\mathcal{M}}T_{x}\mathcal{M}$
stands for the total space of the tangent bundle with natural projection
$\pi(\cdot):T\mathcal{M}\mapsto\mathcal{M}$, and $T_{x}\mathcal{M}=\pi^{-1}(x)$
denotes the tangent space to $\mathcal{M}$ at $x$.

This paper aims to study a time-quasiperiodic second-order system
\begin{gather}
\nabla_{\dot{x}}\dot{x}=f(t\omega,x),\label{eq:qp-sys-f}
\end{gather}
 as well as  its perturbation
\begin{gather}
\nabla_{\dot{x}}\dot{x}=f(t\omega,x)+P(t\omega,x)\dot{x},\label{eq:qp-sys-f-P}
\end{gather}
where $f(\cdot,\cdot):\mathbb{T}^{k}\times\mathcal{M}\mapsto T\mathcal{M}$
is a smooth mapping generating the smooth family of vector fields
$\left\{ f(\varphi,\cdot)\right\} _{\varphi\in\mathbb{T}^{k}}$ on
$\mathcal{M}$ parametrized by points of the standard $k$-dimensional
torus $\mathbb{T}^{k}:=\mathbb{R}^{k}/2\pi\mathbb{Z}^{k}$, $\left\{ P(\varphi,\cdot)\right\} _{\varphi\in\mathbb{T}^{k}}$
is a smooth family of $(1,1)$-tensor fields, and $\omega\in\mathbb{R}^{k}$
is the basic frequency vector with rationally independent components.
Systems of such a kind naturally appear as the Newtonian equations
of motion for holonomic mechanical systems undergoing quasiperiodic
excitations and perturbations which are linearly dependent upon velocity.

E.g., consider a mechanical system in Euclidean space $\mathbb{E}^{N}=\left(\mathbb{R}^{N},(\cdot,\cdot)\right)$
endowed with the inner product $(\cdot,\cdot)$. One can introduce
coordinates $y=(y_{1},\ldots,y_{N})$ in such a way that system's
kinetic energy be represented as
\begin{gather*}
\frac{1}{2}(\dot{y},\dot{y})=\frac{1}{2}\sum_{i=1}^{N}\dot{y}_{i}^{2}.
\end{gather*}
Suppose that after imposing constraints the system's configuration
space turns into an $n$-dimensional submanifold $\mathcal{M}$ embedded
into $\mathbb{E}^{N}$ by means of inclusion map $\iota:\mathcal{M}\hookrightarrow\mathbb{R}^{N}$.
The inner product $(\cdot,\cdot)$ induces on $\mathcal{M}$ the metric
tensor $\mathfrak{g}=\left\langle \cdot,\cdot\right\rangle :=\left(\iota^{\prime}\cdot,\iota^{\prime}\cdot\right)$
, where $\iota^{\prime}:T\mathcal{M}\mapsto\mathbb{E}^{N}$ is the
derivative of $\iota$, and thus the kinetic energy of constrained
system becomes
\begin{gather*}
T(\dot{x})=\frac{1}{2}\left(\iota^{\prime}(x)\dot{x},\iota^{\prime}(x)\dot{x}\right)=\frac{1}{2}\left\langle \dot{x},\dot{x}\right\rangle .
\end{gather*}
 Let $\Phi(t\omega,y,\dot{y})$ be the resultant force acting on the
system and $F(t\omega,x,\dot{x})$ be the generalized force correctly
defined by the relation
\begin{gather*}
\left(\Phi\left(t\omega,\iota(x),\iota^{\prime}(x)\dot{x}\right),\iota^{\prime}(x)\xi\right)=\left\langle F(t\omega,x,\dot{x}),\xi\right\rangle
\end{gather*}
which is required to be true for any vector $\xi\in T_{x}\mathcal{M}$.
It turns out that according to the well known variational principle
of analytical mechanics (see, e.g., \cite{Pars65}) the equation of
motion in coordinate-independent form can be represented as
\begin{gather*}
\nabla_{\dot{x}}\dot{x}=F(t\omega,x,\dot{x}).
\end{gather*}
(In local coordinates $(x_{1},\ldots,x_{n})$, the kinetic energy
has the form $T(\dot{x})=\frac{1}{2}\sum g_{ij}(x)\dot{x_{i}}\dot{x}_{j}$
where $g_{ij}(x)$ are components of metric tensor, and the corresponding
equations are
\begin{gather*}
\ddot{x}_{i}+\sum_{j,l=1}^{n}\Gamma_{jl}^{i}(x)\dot{x}_{j}\dot{x}_{l}=F_{i}(t\omega,x,\dot{x}),\quad i\in\left\{ 1,\ldots,n\right\} ,
\end{gather*}
where $\Gamma_{jl}^{i}(x)$ are the Chistoffel symbols.)

In many cases the dependence of the resultant force upon velocity
$\dot{y}$ is weak and linear. For this reason it is naturally to
consider that the generalize force has the form
\begin{gather*}
F(t\omega,x,\dot{x})=f(t\omega,x)+P(t\omega,x)\dot{x}
\end{gather*}
where $P(\varphi,x)$, in some sense, is small uniformly with respect
to $(\varphi,x)\in\mathbb{T}^{k}\times\mathcal{M}$.

A classical problem for Systems (\ref{eq:qp-sys-f}) and (\ref{eq:qp-sys-f-P})
is whether there exists a quasiperiodic solution with frequency vector
$\omega$, i.e. a solution repesented in the form $x(t)\equiv u(t\omega)$,
where $u(\cdot):\mathbb{T}^{k}\mapsto\mathcal{M}$ is a continuous
mapping. Such a solution will be called $\omega$-quasiperiodic.

In Euclidean configuration space with constant metric tensor, the
above mentioned problem was studied by many authors even in more general
almost periodic case. Non-local existence results for bounded and
almost periodic solutions were obtained under certain monotonicity,
convexity or coercivity conditions using topological principles, methods
of nonlinear analysis, variational approach etc. (see. \cite{BerZha96,Blo88,Blo89,Blo89_1,BloCie97,BloPen01,Car98,Che05,Cie03,Cher74,Kua12,Maw00,Ort09,TruPer86,ZahPar99}).
A detailed enough survey on the problem can be found, e.g., in \cite{Cie03}.

The attempts to extend results of the above papers to systems on Riemannian
manifolds meet essential difficulties, especially in the case of manifolds
where sectional curvature can take positive values. A number of results
in this direction were obtained in \cite{ZahPar99_2,ParRus12,Par14}
by means of variational approach. All these results concern natural
Largangian systems. The Lagrangian density of natural time-quasiperiodic
mechnical system on $\mathcal{M}$ is represented as the difference
o kinetic and potential energy: $L=\left\langle \dot{x},\dot{x}\right\rangle /2-\Pi(t\omega,x)$,
where $\Pi(\cdot,\cdot):\mathbb{T}^{k}\times\mathcal{M}\mapsto\mathbb{R}$.
The corresponding equations of motion has the form~(\ref{eq:qp-sys-f})
where for any $\varphi\in\mathbb{T}^{k}$ the vector field $f(\varphi,\cdot)$
is the gradient of the function $-\Pi(\varphi,\cdot):\mathcal{M}\mapsto\mathbb{R}$.

In the present paper, we obtain a novel results concerning the existence
of bounded as well as quasiperiodic solutions to Systems~(\ref{eq:qp-sys-f})
and (\ref{eq:qp-sys-f-P}). Analogously to the papers \cite{ParRus12,Par14},
the corresponding existence theorems are formulated in terms of $ $
an auxiliary function $U(\cdot)$. In particular, by means of this
function we introduce the notion of $U$-monotonicity for System~(\ref{eq:qp-sys-f}).
The $U$-monotonicity property of the system implies that the associated
variational system with respect to any bounded solution is hyperbolic.
Our results can be regarded as a generalization of those established
in~\cite{ParRus12,Par14}. In contrary to these papers, now we do
not assume $f(\varphi,\cdot)$ necessarily to be the gradient of a
function. Besides, we exploit a version of Wa\.zewski topological
principle instead of variational approach. In such a way we avoid
a cumbersome procedure of transition from generalized quasiperiodic
solutions to classical ones. It should be noted that due to the tools
of global Riemannian geometry we nowhere resorted to the usage of
local coordinates.

The present paper is organized as follows. In Section~\ref{sec:Not-hyp-res},
we formulate our main results, concerning the following issues: (a)
the existence of a solution taking values in a given bounded domain
of configuration space and possessing a bounded derivative; (b) the
hyperbolicity of such a solution; (c) the uniqueness and, as a consequence,
the quasiperiodicity of such a solution. In Section~\ref{sec:Auxiliary-propositions},
a number of important auxiliary propositions are proved, including
the Landau type inequality for the derivative of the bounded solution.
The main theorems are proved in Section~\ref{sec:proofs}. In particular,
here we present an ad hoc proof of quasiperiodicity without referring
to the well known Amerio theorem~\cite{Ame55}. Finally, in Section~\ref{sec:charged particle},
the results obtained are applied to establish conditions under which
the system governing the motion of charged particle in time-quasiperiodic
electric field has a hyperbolic quasiperiodic solution. We also show
that the perturbation of the system by sufficiently small time-quasiperiodic
magnetic field together with the force of friction does not destroy
such a quasiperiodic solution. The admissible magnitude of perturbation
is estimated.

\section{Notations and main results\label{sec:Not-hyp-res}}

In what follows we shall use the following notations: $\mathcal{F}$
is the space of smooth (i.e. infinitely differentiable) real-valued
functions on $\mathcal{M}$; $T_{x}\mathcal{M}$ is the tangent space
at the point $x\in\mathcal{M}$; $\mathcal{T}$ is the space of smooth
vector fields on $\mathcal{M}$; $\left\Vert \cdot\right\Vert :=\sqrt{\left\langle \cdot,\cdot\right\rangle }$
is the norm defined by $\mathfrak{g}$; $\nabla_{\xi}v(x)$ is the
covariant derivative of a vector field $v(\cdot)\in\mathcal{T}$ along
a tangent vector $\xi$ at point $x=\pi(\xi)$; for any fixed $\varphi\in\mathbb{T}^{k}$,
$\nabla_{\xi}f(\varphi,x)$ and $\nabla_{\xi}P(\varphi,x)$ are, respectively,
the covariant derivatives of tensor fields $f(\varphi,\cdot)$ and
$P(\varphi,\cdot)$ along $\xi\in T_{x}\mathcal{M}$; $\nabla f(\varphi,\cdot)$
and $\nabla P(\varphi,\cdot)$ are, respectively, $(1,1)$- and $(2,1)$-tensor
fields such that $\nabla f(\varphi,x)\xi=\nabla_{\xi}f(\varphi,x)$,
$\nabla P(\varphi,x)(\xi,\eta)=\nabla_{\xi}P(\varphi,x)\eta$ for
any $\xi,\eta\in T_{x}\mathcal{M}$; $\nabla U(x)\in T_{x}\mathcal{M}$
and $H_{U}$$(x):T_{x}\mathcal{M}\mapsto T_{x}\mathcal{M}$ are, respectively,
the gradient vector and the Hesse form at $x$ of a function $U(\cdot)\in\mathcal{F}$
(by the definition $\left\langle H_{U}(x)\xi,\eta\right\rangle =\left\langle \nabla_{\xi}\nabla U(x),\eta\right\rangle $
for any $x\in\mathcal{M}$ and any $\xi,\eta\in T_{x}\mathcal{M}$
); if $W(\cdot,\cdot):\mathbb{T}^{k}\times\mathcal{M}\mapsto\mathbb{R}$
is a smooth function, then for any fixed $\varphi\in\mathbb{T}$ the
function $W(\varphi,\cdot)\in\mathcal{F}$ naturally defines the gradient
$\nabla W(\varphi,x)$ and the Hesse form $H_{W}(\varphi,x)$ at the
point $x$.

Let $R$ be the curvature tensor of Levi \textendash{} Civita connection
(defined as in~\cite{GKM68}) $\sigma=\sigma(\xi,\eta)$ be a 2-dimensional
plane spanned on linearly independent vectors $\xi,\eta\in T_{x}\mathcal{M}$.
Then
\begin{gather*}
K_{\sigma}(x):=\frac{\left\langle R(\xi,\eta)\eta,\xi\right\rangle }{\left\Vert \xi\right\Vert ^{2}\left\Vert \eta\right\Vert ^{2}-\left\langle \xi,\eta\right\rangle ^{2}}
\end{gather*}
is the Riemannian curvature in direction $\sigma$ at the point $x\in\mathcal{M}$
(see, e.g. \cite{GKM68}). Denote by $\mathfrak{G}_{x}^{2}$ the Grassmann
manifold of 2-dimensional linear subspaces in $T_{x}\mathcal{M}$
and define

\begin{gather*}
K(x):=\max\left\{ 0,\sup\left\{ K_{\sigma}(x):\sigma\in\mathfrak{G}_{x}^{2}\right\} \right\} .
\end{gather*}

Set
\begin{gather}
M_{f}(x):=\max\left\{ \left\Vert f(\varphi,x)\right\Vert :\varphi\in\mathbb{T}^{k}\right\} ,\label{eq:def-Mf}\\
\Lambda_{P}(x):=\max\left\{ \left\langle P(\varphi,x)\xi,\xi\right\rangle :\varphi\in\mathbb{T}^{k},\;\xi\in T_{x}\mathcal{M},\;\left\Vert \xi\right\Vert =1\right\} \label{eq:def-LP}
\end{gather}
Now let us formulate the results concerning the existence of bounded
solutions to Systems~(\ref{eq:qp-sys-f}) and~(\ref{eq:qp-sys-f-P}).
\begin{thm}
\label{thm:Th-Bound-sol-f}Let the following hypotheses be satisfied:

\textbf{H1:} there exist a function $U(\cdot)\in\mathcal{F}$ and
a bounded domain $\mathcal{D}\subset\mathcal{M}$ such that
\begin{gather}
\lambda_{U}(x):=\min\left\{ \left\langle H_{U}(x)\xi,\xi\right\rangle :\xi\in T_{x}\mathcal{M},\left\Vert \xi\right\Vert =1\right\} >0\quad\forall x\in\mathrm{cl}(\mathcal{D})\label{eq:cond_U}
\end{gather}
and
\begin{gather*}
\min\left\{ \left\langle \nabla U(x),f(\varphi,x)\right\rangle :(\varphi,x)\in\mathbb{T}^{k}\times\mathrm{cl}(\mathcal{D})\right\} <0;
\end{gather*}

\textbf{H2:} the boundary $\partial\mathcal{D}$ of the domain $\mathcal{D}$
is a smooth hypersurface and for any $(\varphi,x)\in\mathbb{T}^{k}\times\partial\mathcal{D}$
there hold the inequalities
\begin{gather*}
\left\langle \nu(x),f(\varphi,x)\right\rangle >0,\quad\lambda_{II}(x)>0
\end{gather*}
where $\nu(x)$ and $\lambda_{II}(x)$ stand, respectively, for the
unite vector of outward normal and the minimal principal curvature
of the boundary%
\footnote{Here the second fundamental tensor for $\partial\mathcal{D}$ with
respect to inward normal vector field is defined in the following
way: $ $$\left\langle II(x)\xi,\eta\right\rangle =\left\langle \nabla_{\xi}\nu(x),\eta\right\rangle \quad\forall\xi,\eta\in T_{x}\partial\mathcal{M}.$%
} at point $x\in\partial\mathcal{D}$, i.e.
\begin{gather*}
\lambda_{II}(x):=\min\left\{ \left\langle \nabla_{\xi}\nu(x),\xi\right\rangle :\xi\in T_{x}\partial\mathcal{M},\left\Vert \xi\right\Vert =1\right\} .
\end{gather*}

Then System(\ref{eq:qp-sys-f-P}) has a solution $x_{\ast}(\cdot):\mathbb{R}\mapsto\mathcal{D}$
such that
\[
\sup_{t\in\mathbb{R}}\left\Vert \dot{x}_{\ast}(t)\right\Vert \le z_{\ast}:=q\zeta_{\ast}(C_{f}C_{U}/q^{2})
\]
 where
\begin{gather}
C_{f}:=\max\left\{ \frac{M_{f}(x)}{\lambda_{U}(x)}:x\in\mathrm{cl}(\mathcal{D})\right\} ,\quad C_{U}:=\max\left\{ \left\Vert \nabla U(x)\right\Vert :\; x\in\mathrm{cl(\mathcal{D})}\right\} ,\label{eq:def-Cf-CU}\\
q:=\sqrt{\max\left\{ -\frac{\left\langle \nabla U(x),f(\varphi,x)\right\rangle }{\lambda_{U}(x)}:\left(\varphi,x\right)\in\mathbb{T}^{k}\times\mathrm{cl}(\mathcal{D})\right\} ,}\label{eq:def-q1}
\end{gather}
 and $\zeta_{\ast}(m)$ stands for  the greatest root of the polynomial
$\zeta\mapsto$$\zeta^{3}-3\zeta+2-3m$.\end{thm}
\begin{rem}
If $\lambda_{U}(x)>0$ and $\left\langle \nabla U(x),f(\varphi,x)\right\rangle $
is non-negative in $\mathbb{T}^{k}\times\mathrm{cl}(\mathcal{D})$,
then System~(\ref{eq:qp-sys-f}) does not have non-constant bounded
on $\mathbb{R}_{+}$ solutions (see Remark~\ref{rem:C-le-0} below).
\end{rem}
The proof of Theorem~\ref{thm:Th-Bound-sol-f} remains correct if
instead of (\ref{eq:def-q1}) we put $q=\sqrt{C_{f}C_{U}}$. Since
the greatest root of the polynomial $\zeta^{3}-3\zeta-1$ does not
exceed $1.88$, then the above estimate for $\left\Vert \dot{x}_{\ast}(t)\right\Vert $
can be replaced by the following one
\begin{gather*}
\sup_{t\in\mathbb{R}}\left\Vert \dot{x}_{\ast}(t)\right\Vert \le1.88\sqrt{C_{f}C_{U}}.
\end{gather*}

\begin{rem}
In \cite{Cie03}, for a system $\ddot{x}=f(t\omega,x)$ in Euclidean
space $\mathbb{E}^{n}$ an estimate for derivative of solution $x(\cdot):\mathbb{R}\mapsto\mathcal{B}_{R}$:=$\left\{ x:\left\Vert x\right\Vert \le R\right\} $
is obtained by means of the Landau inequality. With the Hadamard best
possible constant, this inequality reads as follows
\begin{gather*}
\sup_{t\in\mathbb{R}}\left\Vert \dot{x}(t)\right\Vert \le\sqrt{2\sup_{t\in\mathbb{R}}\left\Vert x(t)\right\Vert \sup_{t\in\mathbb{R}}\left\Vert \ddot{x}(t)\right\Vert }.
\end{gather*}
If we take $U(x):=\left\Vert x\right\Vert ^{2}/2$, then $C_{U}=R$,
$\lambda_{U}=1$, and
\[
\sup_{t\in\mathbb{R}}\left\Vert \ddot{x}(t)\right\Vert \le\max\left\{ \left\Vert f(\varphi,x)\right\Vert :\varphi\in\mathbb{T}^{k},\left\Vert x\right\Vert \le R\right\} =C_{f}.
\]
Thus, in the case of $\mathcal{M}=\mathbb{E}^{n}$, the Landau-Hadamard
inequality yields somewhat better estimate $\sup_{t\in\mathbb{R}}\left\Vert \dot{x}(t)\right\Vert \le\sqrt{2C_{f}C_{U}}$.
\end{rem}
Now let us proceed to the perturbed system. Set
\begin{gather}
l:=\max\left\{ \frac{\Lambda_{P}(x)}{M_{f}(x)}:x\in\mathrm{cl}(\mathcal{D})\right\} ,\label{eq:def-l}\\
p:=\max\left\{ \frac{\left\Vert P^{\ast}(\varphi,x)\nabla U(x)\right\Vert }{\lambda_{U}(x)}:\left(\varphi,x\right)\in\mathbb{T}^{k}\times\mathrm{cl}(\mathcal{D})\right\} \nonumber
\end{gather}
where the conjugate $P^{\ast}(\varphi,x)$ is defined in a standard
way: $\left\langle P(\varphi,x)\xi,\eta\right\rangle =\left\langle \xi,P^{\ast}(\varphi,x)\eta\right\rangle $
for all $\xi,\eta\in T_{x}\mathcal{M}$. Introduce the numbers
\begin{gather}
z_{\pm}:=\frac{p\mbox{\ensuremath{\pm}}\sqrt{4q^{2}+p^{2}}}{2}\label{eq:def-z-pm}
\end{gather}
 (the roots of the polynomial $z\mapsto z^{2}-pw-q^{2}$) and the
function
\begin{gather*}
I(z):=\intop_{z_{+}}^{z}\frac{\left(w^{2}-pw-q^{2}\right)}{lw+1}\mathrm{d}w=\\
\frac{1}{l}\left[\frac{1}{2}w^{2}-\frac{1+pl}{l}w+\frac{1+pl-q^{2}l^{2}}{l^{2}}\ln(1+lw)\right]\Biggl|_{z_{+}}^{z}.
\end{gather*}

\begin{thm}
\label{thm:Th-B-sol-f-P}Let Hypotheses \textbf{H1} and \textbf{H2}
be satisfied, and in addition,
\begin{gather}
4\lambda_{II}(x)\left\langle \nu(x),f(\varphi,x)\right\rangle >\left\Vert P^{\ast}(\varphi,x)\nu(x)\right\Vert ^{2}\quad\forall(\varphi,x)\in\mathbb{T}^{k}\times\partial\mathcal{D}.\label{eq:4lambda-ge}
\end{gather}
 Then System(\ref{eq:qp-sys-f-P}) has a solution $x_{\ast}(\cdot):\mathbb{R}\mapsto\mathcal{D}$
such that \textup{$\sup_{t\in\mathbb{R}}\left\Vert \dot{x}_{\ast}(t)\right\Vert $}
does not exceed the root $z_{\ast}\in[z_{+},\infty)$ of equation
$I(z)=C_{f}C_{U}z_{+}$. \end{thm}
\begin{rem}
\label{rem: Th-B-sol-f-P} Suppose that $l<z_{+}/q^{2}$, and thus
$1+pl-q^{2}l^{2}>0$. Since $I(z_{+})=I^{\prime}(z_{+})=0$, $I^{\prime\prime\prime}(z)>0$
and $I^{\prime\prime}(z_{+})>2q(1+lz_{+})^{-2}$, then
\begin{gather*}
I(z)\ge\frac{1}{2}I^{\prime\prime}(z_{+})(z-z_{+})^{2}\ge\frac{q(z-z_{+})^{2}}{(1+lz_{+})^{2}},
\end{gather*}
and hence
\begin{gather*}
z_{\ast}\le z_{+}+(1+lz_{+})\sqrt{\frac{C_{f}C_{U}z_{+}}{q}.}
\end{gather*}

\end{rem}
The next two results establish conditions which ensures the hyperbolicity
of bounded solutions. Recall the corresponding notion. Let $x(\cdot,\cdot):(t_{1},t_{2})\times(-\sigma,\sigma)\mapsto\mathcal{M}$
be such a smooth mapping that $x(\cdot,s):(t_{1},t_{2})\mapsto\mathcal{M}$
is a solution of~(\ref{eq:qp-sys-f-P}) for any fixed $s\in(-\sigma,\sigma)$,
the number $\sigma>0$ being sufficiently small. Define two tangent
vector fields along the mapping $x(\cdot,\cdot)$:
\begin{gather*}
\dot{x}(t,s):=\frac{\partial x(t,s)}{\partial t},\quad x^{\prime}(t,s):=\frac{\partial x(t,s)}{\partial s}.
\end{gather*}
 Then $\nabla_{\dot{x}}x^{\prime}=\nabla_{x^{\prime}}\dot{x}$ and
\begin{gather*}
\nabla_{x^{\prime}}\nabla_{\dot{x}}\dot{x}-\nabla_{\dot{x}}\nabla_{x^{\prime}}\dot{x}=R(x^{\prime},\dot{x})\dot{x}.
\end{gather*}
Since
\begin{gather*}
\nabla_{x^{\prime}}\nabla_{\dot{x}}\dot{x}=\nabla_{x^{\prime}}\left(f(t\omega,x)+P(t\omega,x)\dot{x}\right)\intertext{then}\nabla_{\dot{x}}^{2}x^{\prime}=\nabla f(t,x)x^{\prime}-R(x^{\prime},\dot{x})\dot{x}+\nabla P(t\omega,x)(x^{\prime},\dot{x})+P(t\omega,x)\nabla_{\dot{x}}x^{\prime}.
\end{gather*}
Put here $s=0$ and denote $\tau(t):=\dot{x}(t,0$). The vector fields
\begin{gather*}
\eta(t):=x^{\prime}(t,0),\quad\zeta(t)=\nabla_{\tau}\eta(t),
\end{gather*}
 along the mapping $x(\cdot):=x(\cdot,0)$ satisfy the first order
linear variational system with respect to solution $x(t)$$:$
\begin{gather*}
\\
\end{gather*}
\begin{gather*}
\begin{split}\nabla_{\tau(t)}\eta & =\zeta\\
\nabla_{\tau(t)}\zeta & =\left[\nabla f\left(t\omega,x\right)\eta-R(\eta,\tau(t))\tau(t)+\nabla P\left(t\omega,x\right)(\eta,\tau(t))+P\left(t\omega,x\right)\zeta\right]_{x=x(t)}.
\end{split}
\end{gather*}
If the solution $x(t)$ is extendable on the whole real axis and the
variational system is exponentially dichotomic on $\mathbb{R}$, then
such a solution is called hyperbolic.
\begin{defn}
We shall say that System~(\ref{eq:qp-sys-f}) is $U$-monotone in
$\mathcal{D}$ if there exists $U(\cdot)\in\mathcal{F}$ satisfying
the inequalities
\begin{gather}
\lambda_{f}(\varphi,x)+\frac{\left\langle \nabla U(x),f(\varphi,x)\right\rangle }{2}>0\quad\forall(\varphi,x)\in\mathbb{T}^{k}\times\mathrm{cl}(\mathcal{D}),\label{eq:lamb_f+nablaU}
\end{gather}

\begin{gather}
\mu_{U}(x)\ge2K(x)\quad\forall x\in\mathcal{D}\label{eq:cond-mu}
\end{gather}
where
\begin{gather*}
\lambda_{f}(\varphi,x):=\min\left\{ \left\langle \nabla f(\varphi,x)\eta,\eta\right\rangle :\eta\in T_{x}\mathcal{M},\left\Vert \eta\right\Vert =1\right\} ,
\end{gather*}
\begin{gather}
\mu_{U}(x):=\min\left\{ \left\langle H_{U}(x)\eta,\eta\right\rangle -\frac{\left\langle \nabla U(x),\eta\right\rangle ^{2}}{2}:\eta\in T_{x}\mathcal{M},\left\Vert \eta\right\Vert =1\right\} .\label{eq:def-mu_U}
\end{gather}
\end{defn}
\begin{rem}
Let $\mathcal{M}=\mathbb{E}^{n}$, $K(x)\equiv0$. The standard monotonicity
condition for a second-order system $\ddot{x}=f(t\omega,x)$ requires
the quadratic form $y\mapsto\left\langle f_{x}^{\prime}(\varphi,x)y,y\right\rangle $
to be positive definite. In such a case, if $x(\cdot)$ is a solution
of the second order system, then the indefinite quadratic form $\left\langle y,z\right\rangle $
in $\mathbb{R}^{n}\times\mathbb{R}^{n}$ has a positive definite derivative
along any solution of variational system $\dot{y}=z,\;\dot{z}=f_{x}^{\prime}(t\omega,x(t))y$
equivalent to second order linear system $\ddot{y}=f_{x}^{\prime}(t\omega,x(t))y$,
and thus the variational system is dichotomic \cite{Sam2002}. $ $If
$\mathcal{M}$ is such a manifold that $K(x)>0$, then one can try
to ensure the $U$-monotonicity by means of appropriate choice of
function $U(\cdot)$. As will be shown below, if (\ref{eq:qp-sys-f})
is $U$-monotone, then the modified indefinite non-degenerate quadratic
form
\begin{gather*}
(\eta,\zeta)\mapsto\left\langle \eta,\zeta\right\rangle +\frac{\left\langle \nabla U(x(t)),\dot{x}(t)\right\rangle \left\Vert \eta\right\Vert ^{2}}{2}
\end{gather*}
 has positive derivative along solutions of variational system
\begin{gather*}
\begin{split}\nabla_{\tau(t)}\eta & =\zeta,\\
\nabla_{\tau(t)}\zeta & =\left[\nabla f\left(t\omega,x\right)\eta-R(\eta,\tau(t))\tau(t)\right]_{x=x(t)}.
\end{split}
\end{gather*}
Thus, this system is hyperbolic. \end{rem}
\begin{thm}
\label{thm:hyperbol} Let the following hypothesis be satisfied

\textbf{H3:} System~(\ref{eq:qp-sys-f}) is $U$-monotone in $\mathcal{D}$.

If $x(\cdot):\mathbb{R}\mapsto\mathcal{D}$ is a solution of System~(\ref{eq:qp-sys-f})
such that $\sup_{t\in\mathbb{R}}\left\Vert \dot{x}(t)\right\Vert <\infty$,
then this solution is hyperbolic.
\end{thm}
The next theorem concerns the perturbed system. It is well known that
sufficiently small perturbations do not destroy the hyperbolic solution
of unperturbed system. With our upproah we are able to establish realistic
bounds for perturbations which preserve the hyperbolicity of solution
contained in $\mathcal{D}$.

Set
\begin{gather}
M_{P}(\varphi,x):=\left\Vert P(\varphi,x)\right\Vert ,\quad M_{U}(x):=\left\Vert \nabla U(x)\right\Vert ,\quad M_{PU}(\varphi,x):=\left\Vert P^{\ast}(\varphi,x)\nabla U(x)\right\Vert ,\label{eq:def-MPU}\\
L_{P}(\varphi,x):=\max\left\{ \left|\left\langle \nabla P(\varphi,x)(\eta,\xi),\eta\right\rangle \right|:\;\xi,\eta\in T_{x}\mathcal{M},\;\left\Vert \xi\right\Vert =\left\Vert \eta\right\Vert =1\right\} .\label{eq:def-LP-1}
\end{gather}

\begin{thm}
\label{thm:hyperbol-P} Let Hypothesis \textbf{H3} be satisfied. If
$x(\cdot):\mathbb{R}\mapsto\mathcal{D}$ is a solution of System~(\ref{eq:qp-sys-f-P})
such that $\sup_{t\in\mathbb{R}}\left\Vert \dot{x}(t)\right\Vert :=Z<\infty$,
and in addition,
\begin{gather}
\lambda_{f}(\varphi,x)+\frac{\left\langle \nabla U(x),f(\varphi,x)\right\rangle }{2}>\sigma(\varphi,x;Z)\quad\forall(\varphi,x)\in\mathbb{T}^{k}\times\mathrm{cl}(\mathcal{D})\label{eq:lambdaf>sigmaZ}
\end{gather}
where
\begin{gather*}
\sigma(\varphi,x;Z):=\frac{\left(M_{U}(x)M_{P}(\varphi,x)+M_{PU}(\varphi,x)+2L_{P}(\varphi,x)\right)Z}{2}+\frac{M_{P}^{2}(\varphi,x)}{4},
\end{gather*}
 then the solution $x(t)$ is hyperbolic.
\end{thm}
Finally, let us present the results on the existence of quasiperiodic
solutions.
\begin{thm}
\label{thm:quasiper-sol-f}Let the Hypotheses \textbf{H1} \textendash{}
\textbf{H3} be satisfied, and in addition, suppose that there holds
the inequality

\begin{gather}
\lambda_{II}(x)+\frac{1}{2}\left\langle \nabla U(x),\nu(x)\right\rangle >0\quad\forall x\in\partial\mathcal{D}.\label{eq:II-cond}
\end{gather}
Then the domain $\mathcal{D}$ contains the unique solution $x_{\ast}(\cdot):\mathbb{R}\mapsto\mathcal{D}$
of System~(\ref{eq:qp-sys-f}). This solution is $\omega$-quasiperiodic
and hyperbolic.$ $ \end{thm}
\begin{rem}
Suppose that there exists a noncritical value $c\in U(\mathcal{M})$
such that $\lambda_{U}(x)>0$ in a connected component $\mathcal{\tilde{D}}$
of sub-level set $U^{-1}(-\infty,c)$. Then the inequality~(\ref{eq:II-cond})
is satisfied for all $x\in\partial\tilde{\mathcal{D}}$.
\end{rem}
For the perturbed system. the corresponding statement is as follows.
\begin{thm}
\label{thm:quasiper-sol-f-P}Let the Hypotheses \textbf{H1} \textendash{}
\textbf{H3} be satisfied, and in addition, suppose that there holds
the inequalities~(\ref{eq:II-cond}), (\ref{eq:4lambda-ge}) and
\begin{gather}
\lambda_{f}(\varphi,x)+\frac{\left\langle \nabla U(x),f(\varphi,x)\right\rangle }{2}>\sigma(\varphi,x;z_{\ast})\quad\forall(\varphi,x)\in\mathbb{T}^{k}\times\mathrm{cl}(\mathcal{D})\label{eq:labdaf-g-sigma}
\end{gather}
where $z_{\ast}$ and $\sigma(\cdot,\cdot;\cdot)$ are defined, respectively,
in Theorem~\ref{thm:Th-B-sol-f-P} and Theorem~\ref{thm:hyperbol-P}.
Then the domain $\mathcal{D}$ contains the unique solution $x_{\ast}(\cdot):\mathbb{R}\mapsto\mathcal{D}$
of System~(\ref{eq:qp-sys-f-P}). This solution is $\omega$-quasiperiodic
and hyperbolic.
\end{thm}

\section{Auxiliary propositions \label{sec:Auxiliary-propositions}}

Propositions~\ref{prop:T=00003Dinfty} \textendash{} \ref{prop:bound-for-dotx-Pert}
below are essential for the proof of Theorems~\ref{thm:Th-Bound-sol-f}
and~\ref{thm:Th-B-sol-f-P}.
\begin{prop}
\label{prop:T=00003Dinfty}Let $\mathcal{D}\subset\mathcal{M}$ be
a bounded domain and let $x(\cdot):(T_{-},T_{+})\mapsto\mathcal{M}$
be a non-extendable solution of System (\ref{eq:qp-sys-f-P}) such
that $x(\cdot):[s,T_{+})\mapsto\mathrm{cl}(\mathcal{D})$ for some
$s\in(T_{-},T_{+})$. Then $T_{+}=\infty$. If $x(\cdot):(T_{-},T_{+})\mapsto\mathrm{cl}(\mathcal{D})$
then $(T_{-},T_{+})=\mathbb{R}$.\end{prop}
\begin{proof}
If we assume that $T_{+}<\infty$, then $\limsup_{t\nearrow T_{+}}\left\Vert \dot{x}(t)\right\Vert =\infty$.
Since
\begin{gather*}
\frac{\mathrm{d}}{\mathrm{d}t}\left\Vert \dot{x}(t)\right\Vert ^{2}=2\left\langle \dot{x},f(t\omega,x)+P(\omega t,x)\dot{x}\right\rangle \Bigl|_{x=x(t)}\le\left[1+2\Lambda_{P}(x(t))\right]\left\Vert \dot{x}(t)\right\Vert ^{2}+M_{f}^{2}(x(t)),
\end{gather*}
then $\left\Vert x(t)\right\Vert ^{2}$ does not exceed the solution
of linear initial problem
\begin{gather*}
\dot{z}=\left[1+2\Lambda_{P}(x(t))\right]z+M_{f}^{2}(x(t)),\quad z(s)=\left\Vert x(s)\right\Vert ^{2}.
\end{gather*}
 Hence $\left\Vert \dot{x}(t)\right\Vert $ is bounded on $[s,T_{+})$
and we arrive at contradiction with our assumption that $T_{+}<+\infty$.

The same arguments can be used to prove that $T_{-}=-\infty$ if $x(\cdot):(-\infty,s]\mapsto\mathrm{cl}(\mathcal{D})$.
\end{proof}
If $\mathcal{M}=\mathbb{R}^{n}$, and the system is $\ddot{x}=f(t\omega,x)$,
then the boundedness of solution on $\mathbb{R}_{+}$ or $\mathbb{R}$
implies the boundedness of its second derivative, respectively, on
$\mathbb{R}_{+}$ or $\mathbb{R}$. In~\cite{Cie03}, the Landau
inequality was used to prove the boundedness of the first derivative
of solution. However, in the case of Riemannian manifold with non-constant
metric tensor, the equation~(\ref{eq:qp-sys-f-P}) written in local
coordinates contains quadratic terms with respect to $\dot{x}$. Thus
the Landau inequality cannot be directly applied to prove the boundedness
of $\dot{x}(\cdot)$. Nevertheless, we have
\begin{prop}
\label{prop:bound-for-dotx-Pert} Suppose that Hypothesis \textbf{H1}
is valid and let $x(\cdot):[s,\infty)\mapsto\mathrm{cl}(\mathcal{D})$
be a solution of System~(\ref{eq:qp-sys-f-P}) such that $\left\Vert \dot{x}(s)\right\Vert \le z_{+}$
(see~(\ref{eq:def-z-pm})). Then $\sup_{t\ge s}\left\Vert \dot{x}(t)\right\Vert \le z_{\ast}$
where $z_{\ast}$ is defined in Theorem~\ref{thm:Th-Bound-sol-f}
if $P(\varphi,x)\equiv0$, and in Theorem~\ref{thm:Th-B-sol-f-P}
otherwise. \end{prop}
\begin{proof}
Define $u(t):=U\left(x(t)\right),\quad v(t):=\dot{u}(t)\equiv\left\langle \nabla U\left(x(t)\right),\dot{x}(t)\right\rangle .$
We have
\begin{gather*}
\dot{v}(t)=\left[\left\langle H_{U}(x)\dot{x},\dot{x}\right\rangle +\left\langle \nabla U(x),f(t\omega,x)+P(t\omega,x)\dot{x}\right\rangle \right]_{x=x(t)}\ge\\
\left[\lambda_{U}(x)\left\Vert \dot{x}\right\Vert ^{2}-\left\Vert P^{\ast}(t\omega,x)\nabla U(x)\right\Vert \left\Vert \dot{x}(t)\right\Vert +\left\langle \nabla U(x),f(t\omega,x)\right\rangle \right]_{x=x(t)}\ge\\
\left[\lambda_{U}(x)\left(\left\Vert \dot{x}\right\Vert ^{2}-p\left\Vert \dot{x}\right\Vert -q^{2}\right)\right]_{x=x(t)}=\left[\lambda_{U}(x)\left(\left\Vert \dot{x}\right\Vert -z_{-}\right)\left(\left\Vert \dot{x}\right\Vert -z_{+}\right)\right]_{x=x(t)}.
\end{gather*}
Let us show that $\left|v(t)\right|\le C_{U}z_{+}$ for all $t\ge s.$
By reasoning ad absurdum, suppose that for some $\delta>0$ the set
\begin{gather*}
\mathcal{T}_{v,\delta}:=\left\{ t\ge s:\left|v(t)\right|>C_{U}(z_{+}+\delta)\right\}
\end{gather*}
is non-empty. Since $\left|v(t)\right|\le C_{U}\left\Vert \dot{x}(t)\right\Vert $,
then $\left\Vert \dot{x}(t)\right\Vert >z_{+}+\delta$ for all $t\in\mathcal{T}_{v}$.
Hence,
\begin{gather*}
\dot{v}(t)\ge\left[\lambda_{U}(x)\left(z_{+}+\delta-\bar{z}_{-}\right)\delta\right]_{x=x(t)}\ge l_{U}\delta^{2}>0\quad\forall t\in\mathcal{T}_{v,\delta}
\end{gather*}
where
\begin{gather}
l_{U}:=\min_{x\in\mathrm{cl}(\mathcal{D})}\lambda_{U}(x).\label{eq:def-lu}
\end{gather}
Thus $v(t)$ does not decrease while $t\in\mathcal{T}_{v,0}$ . Since
$\left|v(s)\right|\le C_{U}\left\Vert \dot{x}(s)\right\Vert \le C_{U}z_{+}$,
then $s\notin\mathcal{T}_{v,\delta}$ and for this reason $v(t)\ge-C_{U}z_{+}$
for all $t\ge s$. On the other hand, if $t_{0}\in\mathcal{T}_{v,\delta}$
then
\begin{gather*}
v(t)\ge v(t_{0})+l_{U}\delta^{2}(t-t_{0})>C_{U}(z_{+}+\delta)
\end{gather*}
while $t>t_{0}$ and $t\in\mathcal{T}_{v,\delta}$. This implies that
$[t_{0},\infty)\subset\mathcal{T}_{v,\delta}$, $v(t)\to\infty$ as
$t\to\infty$, and hence, $u(t)\to\infty$ as $t\to\infty$. We arrive
at contradiction with our assumption that $x(t)\in\mathcal{D}$ for
all $t\ge s$, since $U(\cdot)$ is bounded in $\mathrm{cl}\,\mathcal{D}$
.

Now let us estimate $\left\Vert \dot{x}(t)\right\Vert $. Consider
the nontrivial case where the set
\begin{gather*}
\mathcal{T}:=\left\{ t>s:\left\Vert \dot{x}(t)\right\Vert >z_{+}\right\}
\end{gather*}
 is non-empty. Obviously that any connected component of this set
is an interval $(t_{1},t_{2})$ such that $t_{1}\ge s$, $t_{2}\le+\infty$,
and $\left\Vert \dot{x}(t_{1})\right\Vert =z_{+}$; besides, $\left\Vert \dot{x}(t_{2})\right\Vert =z_{+}$
if $t_{2}<+\infty$, and $ $ $ $$\liminf_{t\to+\infty}\left\Vert \dot{x}(t)\right\Vert =z_{+}$
if $t_{2}=\infty$. In fact, if the last equality were wrong, then
the same arguments as above would lead to unboundedness of $v(t)$.

Since for any $t\in(t_{1},t_{2})$ we have
\begin{gather*}
\left|\frac{\mathrm{d}}{\mathrm{d}t}\left\Vert \dot{x}(t)\right\Vert ^{2}=2\left\langle \dot{x},f(t\omega,x)+P(\omega t,x)\dot{x}\right\rangle \Bigl|_{x=x(t)}\right|\le2\left[M_{f}(x)\left\Vert \dot{x}\right\Vert +\Lambda_{P}(x)\left\Vert \dot{x}\right\Vert ^{2}\right]_{x=x(t)}\quad\Rightarrow\\
\quad\left|\frac{\frac{\mathrm{d}}{\mathrm{d}t}\left\Vert \dot{x}\right\Vert }{M_{f}(x)+\Lambda_{P}(x)\left\Vert \dot{x}\right\Vert }\right|_{x=x(t)}\le1,
\end{gather*}
(see~(\ref{eq:def-Mf}), (\ref{eq:def-LP})), and
\begin{gather*}
\dot{v}(t)\ge\left[\lambda_{U}(x)\left(\left\Vert \dot{x}\right\Vert ^{2}-p\left\Vert \dot{x}\right\Vert -q^{2}\right)\right]_{x=x(t)}>0,
\end{gather*}
then
\begin{gather*}
\left|\frac{\left[\lambda_{U}(x)\left(\left\Vert \dot{x}\right\Vert ^{2}-p\left\Vert \dot{x}\right\Vert -q^{2}\right)\right]\frac{\mathrm{d}}{\mathrm{d}t}\left\Vert \dot{x}\right\Vert }{M_{f}(x)+\Lambda_{P}(x)\left\Vert \dot{x}\right\Vert }\right|_{x=x(t)}\le\dot{v}(t),
\end{gather*}
and finally,
\begin{gather*}
\left|\frac{\left(\left\Vert \dot{x}\right\Vert ^{2}-p\left\Vert \dot{x}\right\Vert -q^{2}\right)\frac{\mathrm{d}}{\mathrm{d}t}\left\Vert \dot{x}\right\Vert }{1+l\left\Vert \dot{x}\right\Vert }\right|_{x=x(t)}\le C_{f}\dot{v}(t)\quad\forall t\in(t_{1},t_{2})
\end{gather*}
(see (\ref{eq:def-Cf-CU}), (\ref{eq:def-l})). This yields
\begin{gather*}
\end{gather*}
\begin{gather*}
-\dot{v}(t)\le\frac{\mathrm{d}}{\mathrm{d}t}I\left(\left\Vert \dot{x}(t)\right\Vert \right)\le\dot{v}(t)\quad\forall t\in(t_{1},t_{2}).
\end{gather*}
 Then for any $t\in(t_{1},t_{2})$ and for any sufficiently small
$\varepsilon>0$ there exists $t_{\varepsilon}\in(t,t_{2})$ such
that $\left\Vert \dot{x}(t_{\varepsilon})\right\Vert =z_{+}+\varepsilon$.
Now

\begin{gather*}
2C_{f}C_{U}z_{+}+\ge v(t_{\varepsilon})-v(t_{1})=\intop_{t_{1}}^{t}\dot{v}(s)\mathrm{d}s+\intop_{t}^{t_{\varepsilon}}\dot{v}(s)\mathrm{d}s\ge\\
\ge\intop_{t_{1}}^{t}\frac{\mathrm{d}}{\mathrm{d}s}I\left(\left\Vert \dot{x}(s)\right\Vert \right)\mathrm{d}s-\intop_{t}^{t_{\varepsilon}}\frac{\mathrm{d}}{\mathrm{d}s}I\left(\left\Vert \dot{x}(s)\right\Vert \right)\mathrm{d}s\ge2I\left(\left\Vert \dot{x}(t)\right\Vert \right)-I(z_{+}+\varepsilon).
\end{gather*}
Letting $\varepsilon\to+0$ we obtain $I\left(\left\Vert \dot{x}(t)\right\Vert \right)\le C_{f}C_{U}z_{+}$
for all $t\in(t_{1},t_{2})$. Since $I(z_{+})=0$ and $I(z)$ monotonically
tends to $+\infty$ on $[z_{+},+\infty)$, then there exists a unique
$z_{\ast}>z_{+}$ such that $I(z_{\ast})=C_{f}C_{U}z_{+}$. This implies
the required estimate for $\left\Vert \dot{x}(t)\right\Vert $.

In the case where $P(\varphi,x)\equiv0$, we have $l=0$, $p=0$,
$z_{+}=-z_{-}=q$, and
\begin{gather*}
I(z)=\frac{z^{3}}{3}-q^{2}z+\frac{2q^{3}}{3}.
\end{gather*}
Hence, $z_{\ast}$ is a solution of equation
\begin{gather*}
z^{3}-3q^{2}z+2q^{3}=3C_{f}C_{U}q.
\end{gather*}
After the substitution $z=q\zeta$, $m=C_{f}C_{U}/q^{2}$ we obtain
the equation
\begin{gather*}
\quad\zeta^{3}-3\zeta+2=3m.
\end{gather*}
 \end{proof}
\begin{rem}
\label{rem:C-le-0}If $P(\varphi,x)\equiv0$ and $\left\langle \nabla U(x),f(\varphi,x)\right\rangle $
is non-negative in $\mathbb{T}^{k}\times\mathrm{cl}(\mathcal{D})$,
then from the inequality $\dot{v}(t)\ge l_{U}\left\Vert \dot{x}\right\Vert ^{2}\ge v(t)/C_{U}$
it follows that System~(\ref{eq:qp-sys-f})does not have non-constant
solutions $x(\cdot):[s,\infty)\mapsto\mathrm{cl}(\mathcal{D})$.
\end{rem}
Let $U(\cdot)\in\mathcal{F}$. $ $Consider the initial value problem

\begin{gather}
\nabla_{x^{\prime}}x^{\prime}=\frac{\left\Vert x^{\prime}\right\Vert ^{2}}{2}\nabla U(x),\quad x(0)=x_{0}:=\pi(\xi),\quad x^{\prime}(0)=\xi\in T\mathcal{M}\quad\left(x^{\prime}=\frac{\mathrm{d}x}{\mathrm{d}s}\right).\label{eq:sys-x(s,xi)}
\end{gather}
Propositions~\ref{prop:nondegen} and \ref{prop:connect-map} below
are essentially exploited in the proof of Theorems~\ref{thm:quasiper-sol-f}
and \ref{thm:quasiper-sol-f-P}.
\begin{prop}
\label{prop:nondegen} Let $\mathcal{D}$ be a domain in $\mathcal{M}$,
$x_{0}\in\mathcal{D}$, $\xi_{0}\in T_{x_{0}}\mathcal{M}$, and let
$x(\cdot,\xi):[0,1]\mapsto\mathrm{cl}(\mathcal{D})$ be the family
of solutions to initial value problem~(\ref{eq:sys-x(s,xi)}) with
parameter $\xi$ ranging in a neighborhood of vector $\xi_{0}$. Suppose
that the function $U(\cdot)$ satisfies in $\mathcal{D}$ the inequality~(\ref{eq:cond-mu}).
Then the derivative of $x(s,\cdot)$ along $T_{x_{0}}\mathcal{M}$,
\begin{gather*}
\frac{\partial}{\partial\xi}\Bigl|_{\xi=\xi_{0}}x(s,\xi):T_{\xi_{0}}\left(T_{x_{0}}\mathcal{M}\right)\cong T_{x_{0}}\mathcal{M}\mapsto T_{x(s,\xi_{0})}\mathcal{M},
\end{gather*}
 is non-degenerate for all $s\in(0,1]$%
\footnote{Recall the well-known fact from the theory of ODE (see, e.g. \cite{Har64}):
if we denote by $I(\xi)$ the interval of existence for non-extendable
solution to (\ref{eq:sys-x(s,xi)}), then the set $\mathcal{E}:=\left\{ (s,\xi):\xi\in T\mathcal{M},s\in I(\xi)\right\} $
is open in $\mathbb{R}\times T\mathcal{M}$ and the mapping $x(\cdot,\cdot):\mathcal{E}\mapsto\mathcal{M}$
is smooth. \label{fn:1}%
}\end{prop}
\begin{proof}
For a given smooth curve $\xi(\cdot):(-\sigma,\sigma)\mapsto T_{x_{0}}\mathcal{M}$,
where $\sigma>0$ and $\xi(0)=\xi_{0}$, construct the mapping
\begin{gather*}
x(\cdot,\xi(\cdot)):I\times(-\sigma,\sigma)\mapsto\mathcal{M}.
\end{gather*}
 Define the following two tangent vector fields along this mapping
\begin{gather*}
Y(s,r):=\frac{\partial x(s,\xi(r))}{\partial r},\quad Z(s,r)=\frac{\partial x(s,\xi(r))}{\partial s}.
\end{gather*}
Then $\nabla_{Y}Z=\nabla_{Z}Y$ and
\begin{gather*}
\nabla_{Y}\nabla_{Z}Z-\nabla_{Z}\nabla_{Y}Z=R(Y,Z)Z.
\end{gather*}
Since
\begin{gather*}
\nabla_{Y}\nabla_{Z}Z=\nabla_{Y}\left(\frac{\left\Vert Z\right\Vert ^{2}}{2}\nabla U(x)\right),\intertext{then}\nabla_{Z}^{2}Y=\left\langle \nabla_{Y}Z,Z\right\rangle \nabla U(x)+\frac{\left\Vert Z\right\Vert ^{2}}{2}H_{U}(x)Y-R(Y,Z)Z.
\end{gather*}
Put here $r=0$ and denote $\bar{x}(s):=x(s,\xi_{0})$, $\tau(s):=Z(s,0)\equiv\bar{x}^{\prime}(s)$.
We see that the vector fields
\begin{gather*}
\eta(s):=Y(s,0),\quad\zeta(s)=\nabla_{\tau}\eta(s),
\end{gather*}
 along the mapping $\bar{x}(\cdot)$ satisfy the first order system
in variations:
\begin{gather*}
\nabla_{\tau}\eta=\zeta,\\
\nabla_{\tau}\zeta=\left\langle \zeta,\tau\right\rangle \nabla U+\frac{\left\Vert \tau\right\Vert ^{2}}{2}H_{U}(\bar{x})\eta-R(\eta,\tau)\tau
\end{gather*}
We have to show that $\eta(s)\ne0$ for all $s\in(0,1]$ once $\xi^{\prime}(0)\ne0$.
Since
\begin{gather*}
\frac{\mathrm{d}}{\mathrm{d}s}\frac{\left\Vert \eta\right\Vert ^{2}}{2}=\left\langle \nabla_{\tau}\eta,\eta\right\rangle =\left\langle \eta,\zeta\right\rangle
\end{gather*}
then
\begin{gather*}
\frac{\mathrm{d}^{2}}{\mathrm{d}s^{2}}\frac{\left\Vert \eta\right\Vert ^{2}}{2}=\frac{\mathrm{d}}{\mathrm{d}s}\left\langle \eta,\zeta\right\rangle =\left\langle \zeta,\zeta\right\rangle +\left\langle \eta,\left\langle \zeta,\tau\right\rangle \nabla U(\bar{x})+\frac{\left\Vert \tau\right\Vert ^{2}}{2}H_{U}(\bar{x})\eta-R(\eta,\tau)\tau\right\rangle =\\
\left\Vert \zeta\right\Vert ^{2}+\left\langle \nabla U(\bar{x}),\eta\right\rangle \left\langle \zeta,\tau\right\rangle +\frac{\left\Vert \tau\right\Vert ^{2}}{2}\left\langle H_{U}(\bar{x})\eta,\eta\right\rangle -K_{\sigma(\eta,\tau)}(\bar{x})\left[\left\Vert \eta\right\Vert ^{2}\left\Vert \tau\right\Vert ^{2}-\left\langle \eta,\tau\right\rangle ^{2}\right]\ge\\
\left\Vert \zeta\right\Vert ^{2}-\left\Vert \zeta\right\Vert \left\Vert \tau\right\Vert \left\langle \nabla U(\bar{x}),\eta\right\rangle +\left\Vert \tau\right\Vert ^{2}\left[\frac{\left\langle H_{U}(\bar{x})\eta,\eta\right\rangle }{2}-K(\bar{x})\left\Vert \eta\right\Vert ^{2}\right].
\end{gather*}
The condition~(\ref{eq:cond-mu}) yields  that $\frac{\mathrm{d}^{2}}{\mathrm{d}s^{2}}\left\Vert \eta(s)\right\Vert ^{2}\ge0$.

Since
\begin{gather*}
x(0,\xi(r))=x_{0},\quad\frac{\partial}{\partial s}\Bigl|_{s=0}x(s,\xi(r))=\xi(r),
\end{gather*}
then $\eta(0)=0$. From this it follow that the horizontal component
of vector $\eta^{\prime}(0)$ (with respect to the Levi-Civita connection)
vanishes and we otain
\begin{gather*}
\quad\zeta(0)=\nabla_{\tau}\eta\bigl|_{s=0}=\frac{\partial}{\partial s}\Bigl|_{s=0}\frac{\partial}{\partial r}\Bigl|_{r=0}x(s,\xi(r))+O(\left\Vert \eta\right\Vert )=\xi^{\prime}(0).
\end{gather*}
 Hence, if $\xi_{0}^{\prime}:=\xi^{\prime}(0)\ne0$, then
\begin{gather*}
\left\Vert \eta(0)\right\Vert ^{2}=0,\quad\frac{\mathrm{d}}{\mathrm{d}s}\Bigl|_{s=0}\left\Vert \eta(s)\right\Vert ^{2}=0,\quad\frac{\mathrm{d}^{2}}{\mathrm{d}s^{2}}\Bigl|_{s=0}\left\Vert \eta(s)\right\Vert ^{2}=2\left\Vert \zeta(0)\right\Vert ^{2}=2\left\Vert \xi_{0}^{\prime}\right\Vert ^{2}>0.
\end{gather*}
This implies that $\left\Vert \eta(s)\right\Vert >0$ for all $s\in(0,1]$.\end{proof}
\begin{prop}
\label{prop:connect-map}Suppose that a function $U(\cdot)\in\mathcal{F}$
in a bounded domain $\mathcal{D}$ with smooth boundary satisfies
the inequalities~(\ref{eq:cond-mu}), (\ref{eq:II-cond}), and let
$x(\cdot,\xi)$ be the solution of initial value problem (\ref{eq:sys-x(s,xi)}).
Then for any $\left\{ x_{0},x_{1}\right\} \subset\mathrm{cl}(\mathcal{D})$
there exists $\xi(x_{0},x_{1})\in T_{x_{0}}\mathcal{M}$ such that
$x(s,\xi(x_{0},x_{1}))\in\mathcal{D}$ for all $s\in(0,1)$, and $x(1,\xi(x_{0,}x_{1}))=x_{1}$.
Moreover,
\begin{gather*}
\left\Vert \xi(x_{0},x_{1})\right\Vert \le d\quad\forall\left\{ x_{0},x_{1}\right\} \subset\mathrm{cl}(\mathcal{D})
\end{gather*}
 where
\begin{gather*}
d:=\frac{C_{U}\mathrm{e}^{U^{\ast}-U_{\ast}}+\sqrt{\left(C_{U}\mathrm{e}^{U^{\ast}-U_{\ast}}\right)^{2}+2l_{U}\mathrm{e}^{U^{\ast}-U_{\ast}}\left(U^{\ast}-U_{\ast}\right)}}{l_{U}},
\end{gather*}
\begin{gather}
U_{\ast}=\min\left\{ U(x):x\in\mathrm{cl}(D)\right\} ,\quad U^{\ast}=\max\left\{ U(x):x\in\mathrm{cl}(D)\right\} .\label{eq:U^*-U_*}
\end{gather}
\end{prop}
\begin{proof}
Let us fix $x_{0}\in\mathcal{D}$ arbitrarily and define the set
\begin{gather*}
\Xi=\left\{ \xi\in T_{x_{0}}\mathcal{M}:x(s,\xi)\in\mathcal{D}\;\forall s\in[0,1]\right\} .
\end{gather*}
This set is non-empty and open in $T_{x_{0}}\mathcal{M}$. In fact,
$0\in\Xi$ and if $\xi_{0}\in\Xi$ then for all $\xi\in T_{x_{0}}\mathcal{M}$
sufficiently close to $\xi_{0}$ the solution $x(s,\xi)$ is defined
on $\text{[0,1]}$ and takes values in $\mathcal{D}$ (see the footnote
\vref{fn:1}). This means that a small neighborhood of $\xi_{0}\in\Xi$
is contained in $\Xi$. By Proposition~\ref{prop:nondegen} the mapping
\begin{gather*}
X(\cdot):=x(1,\cdot):\Xi\mapsto\mathcal{D}
\end{gather*}
 has non-degenerate derivative $X^{\prime}(\xi)$ at each point $\xi\in\Xi$.
Hence, this mapping is a local diffeomorphism and for this reason
the set $\mathcal{X}:=X\left(\Xi\right)$ is an open subset of $\mathcal{D}$.

To show that $\mathcal{X}=\mathcal{D}$ it remains to prove that the
set $\mathcal{X}$ is closed in $\mathcal{D}$. If we suppose the
opposite to be true, then there exists a sequence $\left\{ x_{k}\right\} \subset\mathcal{X}$
convergent to $x_{\ast}\in\mathcal{D}\setminus\mathcal{X}$. By the
definition of $\mathcal{X}$, there also exists a sequence $\left\{ \xi_{k}\right\} \subset\Xi$
such that $x_{k}=x(1,\xi_{k})$.

Let us show that the sequence $\left\{ \xi_{k}\right\} $ is bounded.
Observe that for any $\xi\in\Xi$ and $s\in[0,1]$ we have
\begin{gather*}
\left[\frac{\mathrm{d}}{\mathrm{d}s}\left(\left\Vert x^{\prime}\right\Vert ^{2}\mathrm{e}^{-U(x)}\right)=\left\langle x^{\prime},\nabla U(x)\right\rangle \left\Vert x^{\prime}\right\Vert ^{2}-\left\Vert x^{\prime}\right\Vert ^{2}\left\langle x^{\prime},\nabla U(x)\right\rangle =0\right]_{x=x(s,\xi)}
\end{gather*}
Hence,
\begin{gather}
\left\Vert x^{\prime}(s,\xi)\right\Vert ^{2}=\left\Vert \xi\right\Vert ^{2}\exp\left(U(x_{0})-U(x(s,\xi))\right).\label{eq:cons-low}
\end{gather}
Since
\begin{gather*}
\left[\frac{\mathrm{d}^{2}}{\mathrm{d}s^{2}}U(x)=\left\langle H_{U}(x)x^{\prime},x^{\prime}\right\rangle +\left\langle \nabla U(x),\nabla_{x^{\prime}}x^{\prime}\right\rangle =\left\langle H_{U}(x)x^{\prime},x^{\prime}\right\rangle +\frac{\left\Vert x^{\prime}\right\Vert ^{2}}{2}\left\Vert \nabla U(x)\right\Vert ^{2}\right]_{x=x(s,\xi)},
\end{gather*}
then by the Taylor formula there exists $\theta_{k}\in(0,1)$ such
that
\begin{gather*}
U(x_{k})=U(x_{0})+\left\langle \nabla U(x_{0}),\xi_{k}\right\rangle +\frac{1}{2}\left[\left\langle H_{U}(x)x^{\prime},x^{\prime}\right\rangle +\frac{\left\Vert x^{\prime}\right\Vert ^{2}}{2}\left\Vert \nabla U(x)\right\Vert ^{2}\right]_{x=x(\theta_{k},\xi_{k})}.
\end{gather*}
The condition~(\ref{eq:cond_U}) yields
\begin{gather*}
U(x_{k})\ge U(x_{0})-\left|\left\langle \nabla U(x_{0}),\xi_{k}\right\rangle \right|+\frac{l_{U}}{2}\left\Vert \xi_{k}\right\Vert ^{2}\exp\left(U(x_{0})-U(x(\theta_{k},\xi_{k}))\right)
\end{gather*}
(see~(\ref{eq:def-lu})). Now obviously the sequence $\left\{ \left\Vert \xi_{k}\right\Vert \right\} $
is bounded and without loss of generality, one can regard that $\xi_{k}\to\xi_{\ast}\in T_{x_{0}}\mathcal{M}\setminus\Xi$.
Since $x_{\ast}=x(1,\xi_{\ast})\in\mathcal{D}$, then there is $s_{\ast}\in(0,1)$
such that $x(s,\xi_{\ast})\in\mathcal{D}$ for all $s\in(0,s_{\ast})$
but $y_{\ast}:=x(s_{\ast},\xi_{\ast})\in\partial\mathcal{D}$.

The boundary $\partial\mathcal{D}$ near $y_{\ast}$ can be defined
by zero-level set of a function. More precisely, there exist a neighborhood
$\mathcal{U}$ of $y_{\ast}$ and a function $G(\cdot)\in\mathcal{F}$
such that $\nabla G(x)\ne0$ in $\mathcal{U}$, $G^{-1}(0)=\partial\mathcal{D}\cap\mathcal{U}$,
$G(x)>0$ in $\mathcal{U}\cap\left(\mathcal{M\setminus\mathrm{cl}}(\mathcal{D})\right)$
and $G(x)<0$ in $\mathcal{U}\cap\mathcal{D}$. Besides,
\begin{gather*}
\nu(x)=\frac{1}{\left\Vert \nabla G(x)\right\Vert }\nabla G(x),\quad\left\langle II(x)\xi,\xi\right\rangle =\frac{1}{\left\Vert \nabla G(x)\right\Vert }\left\langle H_{G}(x)\xi,\xi\right\rangle ,\quad\xi\in T_{x}\partial\mathcal{D}.
\end{gather*}
Now for sufficiently small $\delta>0$ we have
\begin{gather}
g(s):=G(x(s,\xi_{\ast}))<0\quad\forall s\in(s_{\ast}-\delta,s_{\ast}),\quad g(s_{\ast})=G(x(s_{\ast},\xi_{\ast}))=0.\label{eq:g(s)<0}
\end{gather}
 Obviously
\begin{gather*}
g^{\prime}(s_{\ast}):=\frac{\mathrm{d}}{\mathrm{d}s}\Bigl|_{s=s_{\ast}}G(x(s_{\ast},\xi_{\ast}))\ge0.
\end{gather*}
The case where $g^{\prime}(s_{\ast})>0$ is impossible. In fact, in
such a case there would exist $s^{\prime}\in(s_{\ast},1)$ such that
$G(x(s^{\prime},\xi))>0$ and than
\[
G(x(s^{\prime},\xi_{k}))>0\quad\Rightarrow\quad x(s^{\prime},\xi_{k})\not\in\mathcal{D}
\]
 for all sufficiently large natural $k$. Thus, $g^{\prime}(s_{\ast})=0$.
Now observe that
\begin{gather*}
g^{\prime\prime}(s_{\ast})=\left[\left\langle H_{G}(x)x^{\prime},x^{\prime}\right\rangle +\left\langle \nabla G(x),\nabla_{x^{\prime}}x^{\prime}\right\rangle \right]_{x=x(s_{\ast},\xi_{\ast})}=\\
=\left[\left\langle H_{G}(x)x^{\prime},x^{\prime}\right\rangle +\frac{\left\Vert x^{\prime}\right\Vert ^{2}}{2}\left\langle \nabla G(x),U(x)\right\rangle \right]_{x=x(s_{\ast},\xi_{\ast})}.
\end{gather*}
Since $\xi_{\ast}\ne0$ then on account of~(\ref{eq:cons-low}) we
have $x^{\prime}(s_{\ast},\xi_{\ast})\ne0$ and the condition~(\ref{eq:II-cond})
implies that $g^{\prime\prime}(s_{\ast})>0$. But then $g(\cdot)$
reaches its strict local minimum at the point $s=s_{\ast}$, and this
produces a contradiction with (\ref{eq:g(s)<0}).

Thus we have proved that $\mathcal{X}$ is an open-close subset of
open set $\mathcal{D}$. This implies that $\mathcal{X}=\mathcal{D}$.
Now we can assert that for any $\left\{ x_{0},x_{1}\right\} \subset\mathcal{D}$
there exists $\xi(x_{0},x_{1})\in T_{x_{0}}\mathcal{M}$ such that
$x(s,\xi(x_{0},x_{1}))\in\mathcal{D}$ for all $s\in[0,1]$, and $x(1,\xi(x_{0,}x_{1}))=x_{1}$.

By repeating the same arguments as above, we obtain the inequality
\begin{gather*}
U(x_{1})\ge U(x_{0})-\left\Vert \nabla U(x_{0})\right\Vert \left\Vert \xi\right\Vert +\frac{l_{U}}{2}\left\Vert \xi\right\Vert ^{2}\exp\left(U(x_{0})-U(x(\theta,\xi))\right)
\end{gather*}
 with $\xi=\xi(x_{0},x_{1})$ and some $\theta=\theta(x_{0},x_{1})\in(0,1)$.
Hence,
\begin{gather*}
\frac{l_{U}}{2}\mathrm{e}^{U_{\ast}-U^{\ast}}\left\Vert \xi\right\Vert ^{2}-C_{U}\left\Vert \xi\right\Vert -U^{\ast}+U_{\ast}\le0\quad\Rightarrow\quad\left\Vert \xi\right\Vert \le d.
\end{gather*}
 Now let $\left\{ x_{0}^{\ast},x_{1}^{\ast}\right\} \subset\mathrm{cl}(\mathcal{D})$.
One can define sequences $\left\{ x_{i}^{k}\right\} \subset\mathcal{D}$,
$x_{i}^{k}\to x_{i}^{\ast}$, $k\to\infty$, $i\in\left\{ 0,1\right\} $
such that $\xi\left(x_{0}^{k},x_{1}^{k}\right)\to\xi^{\ast}\in T_{x_{0}^{\ast}}\mathcal{M}$,
$\left\Vert \xi^{\ast}\right\Vert \le d$. Then $x(1,\xi^{\ast})=x_{1}^{\ast}$
and $x(s,\xi_{\ast})\in\mathrm{cl}(\mathcal{D})$ for all $s\in[0,1]$.
But actually the above arguments concerning the function $g(s)$ allow
us to assert that there is no point $s_{\ast}\in(0,1)$ such that
$x(s_{\ast},\xi_{\ast})\in\partial\mathcal{D}$.
\end{proof}

\section{Proofs of theorems\label{sec:proofs}}

\subsection*{Proofs of Theorems~\ref{thm:Th-Bound-sol-f} and \ref{thm:Th-B-sol-f-P}}

We proceed straight to the proof of Theorem~\ref{thm:Th-B-sol-f-P}.

Put $\bar{f}(x):=(2\pi)^{-k}\intop_{\mathbb{T}^{k}}f(\varphi,x)\mathrm{d}\varphi$
and observe that $\left\langle \nu(x),\bar{f}(x)\right\rangle >0$
for all $x\in\partial\mathcal{D}$. For $s\subset\mathbb{R}$ and
$x\in\mathcal{M}$ let $t\mapsto X_{s}^{t}(x)$, $t\in I(s,x)\subset\mathbb{R},$
be the non-extendable solution of (\ref{eq:qp-sys-f-P}) satisfying
the initial conditions
\begin{gather*}
X_{s}^{s}(x)=x,\quad\dot{X_{s}^{s}}(x)=\epsilon\bar{f}(x)\quad\left(\dot{X_{s}^{t}}(x):=\frac{\partial}{\partial t}X_{s}^{t}(x)\right)
\end{gather*}
where $\epsilon>0$ is small enough to ensure that $\left\Vert \epsilon\bar{f}(x)\right\Vert \le z_{+}$.
Hence, $\left\Vert \dot{X_{s}^{s}}(x)\right\Vert $ satisfies the
same inequality as $\left\Vert \dot{x}(s)\right\Vert $ in Proposition~\ref{prop:bound-for-dotx-Pert}.
Let us show that there exists $x_{0,s}\in\mathcal{D}$ such that $X_{s}^{t}(x_{0,s})\in\mathcal{D}$
 for all $t\ge s$. We shall exploit ideas of Wa\.zewski topological
principle. By reasoning ad absurdum, suppose that such a $x_{0,s}$
does not exist. Then
\begin{gather*}
T(x):=\sup\left\{ T>s:X_{s}^{t}(x)\in\mathcal{D}\;\forall t\in[s,T]\right\} <\infty\quad\forall x\in\mathcal{D}.
\end{gather*}
Obviously, $y(x):=X_{s}^{T(x)}(x)\in\partial\mathcal{D}$. By the
same arguments as in the proof of Proposition~\ref{prop:connect-map},
there exists a neighborhood $\mathcal{U}$ of $y(x)$ and a function
$G(\cdot)\in\mathcal{F}(\mathcal{M})\mapsto\mathbb{R}$ such that
$\nabla G(x)\ne0$ in $\mathcal{U}$, $G^{-1}(0)=\partial\mathcal{D}\cap\mathcal{U}$,
$G(x)>0$ in $\mathcal{U}\cap\left(\mathcal{M\setminus\mathrm{cl}}(\mathcal{D})\right)$,
$G(x)<0$ in $\mathcal{U}\cap\mathcal{D}$, and on account of $\lambda_{II}(x)>0$
the Hesse form $H_{G}(x)$ is positive definite. There also exists
$\delta(x)>0$ such that
\begin{gather}
G\left(X_{s}^{t}(x)\right)<0\quad\forall t\in[T(x)-\delta(x),T(x)),\quad G\left(X_{s}^{T(x)}(x)\right)=0.\label{eq:G<0}
\end{gather}
Obviously that
\begin{gather*}
\frac{\partial}{\partial t}\Bigr|_{t=T(x)}G\left(X_{s}^{t}(x)\right)=\left\langle \nabla G\left(y(x)\right),\dot{X}_{s}^{T(x)}(x)(x)\right\rangle \ge0,
\end{gather*}
and since there holds the inequality~(\ref{eq:4lambda-ge}), then
\begin{gather*}
\frac{\partial^{2}}{\partial t^{2}}\Bigr|_{t=T(x)}G\left(X_{s}^{t}(x)\right)=\\
\left[\left\langle H_{G}\left(y\right)\xi,\xi\right\rangle +\left\langle \nabla G\left(y\right),f(t\omega,y)+P(t\omega,y)\xi\right\rangle \right]\Bigl|_{y=y(x),\xi=\dot{X}_{s}^{T(x)}(x)}=\\
\left\Vert \nabla G(x)\right\Vert \left[\left\langle II(y)\xi,\xi\right\rangle +\left\langle \nu(y),f(t\omega,y)+P(t\omega,y)\xi\right\rangle \right]\Bigl|_{y=y(x),\xi=\dot{X}_{s}^{T(x)}(x)}\ge\\
\left\Vert \nabla G(x)\right\Vert \left[\lambda_{II}(y)\left\Vert \xi\right\Vert ^{2}-\left\Vert P^{\ast}(t\omega,y)\nu(y)\right\Vert \left\Vert \xi\right\Vert +\left\langle \nu(y),f(t\omega,y)\right\rangle \right]\Bigl|_{y=y(x),\xi=\dot{X}_{s}^{T(x)}(x)}>0.
\end{gather*}
 But now
\begin{gather}
\frac{\partial}{\partial t}\Bigr|_{t=T(x)}G\left(X_{s}^{t}(x)\right)>0.\label{eq:dG/dt>0}
\end{gather}
 In fact, if the derivative in the right hand side of (\ref{eq:dG/dt>0})
were zero, then the function $t\mapsto G\left(X_{s}^{t}(x)\right)$
would achieve a strict local minimum at $t=T(x)$. But this is impossible
on account of~(\ref{eq:G<0}).

The inequality(\ref{eq:dG/dt>0}) together with the inverse function
theorem implies that $T(\cdot):\mathcal{D}\mapsto\mathbb{R}$ is smooth.
Let now $y\in\mathcal{U}\cap\partial\mathcal{D}$. Then
\[
G\left(X_{s}^{s}(y)\right)=0,\quad\frac{\partial}{\partial t}\Bigl|_{t=s}G(X_{s}^{t}(y))=\left\Vert \nabla G(y)\right\Vert \left\langle \nu(y),\bar{f}(y)\right\rangle >0.
\]
Hence, the function $T(\cdot)$ is also smooth in a small neighborhood
of $y$ and $T(y)=s$. It follows from the above that the mapping
$\rho(\cdot,\cdot):\mathrm{cl}(\mathcal{D})\times[0,1]\mapsto\mathrm{cl}(\mathcal{D})$
defined by the formula
\begin{gather*}
\mathrm{cl}(\mathcal{D})\times[0,1]\ni\;(x,u)\mapsto X_{s}^{uT(x)+(1-u)s}(x)\in\mathrm{cl}(\mathcal{D})
\end{gather*}
is a deformation retraction of $\mathrm{cl}(\mathcal{D})$ onto $\partial\mathcal{D}$.
We reach a contradiction, since the boundary of compact manifold $\mathrm{cl}(\mathcal{D})$
cannot be a retract of $\mathrm{cl}(\mathcal{D})$.

Thus we have proved that $x_{0,s}$ does exist. Now from Propositions~\ref{prop:T=00003Dinfty}
and \ref{prop:bound-for-dotx-Pert} it follows that the solution defined
by $x_{s}(t):=X_{s}^{t}(x_{0,s})$ satisfies
\begin{gather}
x_{s}(\cdot):[s,\infty)\mapsto\mathcal{D},\quad\sup_{t\ge s}\left\Vert \dot{x}_{s}(t)\right\Vert \le z_{\ast}.\label{eq:bounds-for-x-dotx}
\end{gather}
Consider the sequence $\left\{ x_{-i}(\cdot):[-i,\infty)\mapsto\mathcal{D}\right\} _{i\in\mathbb{N}}$.
Obviously that $x_{-i}(\cdot)$ is the solution of~(\ref{eq:qp-sys-f-P})
satifying the initial conditions
\begin{gather*}
x\bigl|_{t=0}=x_{-i}(0),\quad\dot{x}\bigl|_{t=0}=\dot{x}_{-i}(0).
\end{gather*}
 From~(\ref{eq:bounds-for-x-dotx}) it follows that there exists
a sub-sequence $i_{j}\to\infty$ , $j\to\infty$, such that the sequence
$\left\{ \left(x_{-i_{j}}(0),\dot{x}_{-i_{j}}(0)\right)\right\} _{j\in\mathbb{N}}$
converges to a point $(x_{0},\xi_{0})$ such that $x_{0}\in\mathrm{cl}(\mathcal{D})$,
$\xi_{0}\in T_{x_{0}}\mathcal{M}$, $\left\Vert \xi_{0}\right\Vert \le C_{\ast}$.
Now it is not hard to see that the non-extendable solution $x_{\ast}(\cdot)$
of (\ref{eq:qp-sys-f-P}) satisfying the initial conditions
\begin{gather*}
x_{\ast}(0)=x_{0},\quad\dot{x}_{\ast}(0)=\xi_{0}
\end{gather*}
 is defined on the whole real line and satisfies the conditions
\begin{gather*}
x_{\ast}(\cdot):\mathbb{R}\mapsto\mathrm{cl}(\mathcal{D}),\quad\sup_{t\in\mathbb{R}}\left\Vert \dot{x}(t)\right\Vert \le z_{\ast}.
\end{gather*}
In fact, if this were not true, then there would exist $t^{\prime}\in\mathbb{R}$
such that either $x_{\ast}(t^{\prime})\in\mathcal{M}\setminus\mathrm{cl}(\mathcal{D})$
or $\left\Vert \dot{x}_{\ast}(t^{\prime})\right\Vert >z_{\ast}$,
and then, respectively, either $x_{i_{j}}(t^{\prime})\notin\mathcal{D}$
or $\left\Vert \dot{x}_{i_{j}}(t^{\prime})\right\Vert >z_{\ast}$
for all sufficiently large $j$. This is impossible on account of~(\ref{eq:bounds-for-x-dotx}).

The same arguments as above (see also \cite{Cie03}) allows us to
show that $x_{\ast}(\cdot):\mathbb{R}\mapsto\mathcal{D}$. In fact,
if there were a moment $t_{0}$ such that $x_{\ast}(t_{0})\in\partial\mathcal{D}$,
then
\begin{gather*}
G(x_{\ast}(t_{0}))=0,\quad\frac{\mathrm{d}}{\mathrm{d}t}\Bigl|_{t=t_{0}}G(x_{\ast}(t))=0,\quad\frac{\mathrm{d^{2}}}{\mathrm{d}t^{2}}\Bigl|_{t=t_{0}}G(x_{\ast}(t))>0.
\end{gather*}
This implies that the function $G\left(x_{\ast}(\cdot)\right)$ achieves
a strict local minimum at $t=t_{0}$, and hence $\left\{ t\in\mathbb{R}:G(x_{\ast}(t))>0\right\} \ne\varnothing$.
We reach a contradiction. Q.E.D.

\subsection*{Proofs of Theorems~\ref{thm:hyperbol} and~\ref{thm:hyperbol-P} }

We start with the proof of Theorem \ref{thm:hyperbol-P}. Let $x(\cdot):\mathbb{R}\mapsto\mathcal{D}$
be a solution of System~(\ref{eq:qp-sys-f-P}) such that $\sup_{t\in\mathbb{R}}\left\Vert \dot{x}(t)\right\Vert :=Z<\infty$.
Put $\tau(t)=\dot{x}(t)$, introduce the non-degenerate quadratic
form
\begin{gather*}
Q(\eta,\zeta;t)=\left\langle \eta,\zeta\right\rangle +\frac{\left\Vert \eta\right\Vert ^{2}}{2}\left\langle \nabla U(x(t)),\tau(t)\right\rangle ,\quad\eta,\zeta\in T_{x(t)}\mathcal{M}
\end{gather*}
 and find its derivative along solution of the linear variational
system with respect to $x(t)$:
\begin{gather*}
\dot{Q}(\eta,\zeta;t)=\left\Vert \zeta\right\Vert ^{2}+\left\langle \eta,\nabla f\eta\right\rangle -\left\langle R(\eta,\tau)\tau,\eta\right\rangle +\left\langle \nabla P(\eta,\tau),\eta\right\rangle +\left\langle P\zeta,\eta\right\rangle +\\
\left\langle \eta,\zeta\right\rangle \left\langle \nabla U,\tau\right\rangle +\frac{\left\Vert \eta\right\Vert ^{2}}{2}\left(\left\langle H_{U}\tau,\tau\right\rangle +\left\langle \nabla U,f+P\tau\right\rangle \right)
\end{gather*}
For the sake of simplifying the calculations, here and below we do
not show explicitly the arguments $\varphi=t\omega$, $x=x(t)$. Taking
into account~(\ref{eq:def-mu_U}), (\ref{eq:def-MPU}), (\ref{eq:def-LP-1}),
(\ref{eq:def-LP-1}) we have
\begin{gather*}
\dot{Q}(\eta,\zeta;t)\ge\\
\ge\left\Vert \zeta\right\Vert ^{2}-\left\Vert \eta\right\Vert \left\Vert \zeta\right\Vert \left(\left|\left\langle \nabla U,\tau\right\rangle \right|+M_{P}\right)+\left\Vert \eta\right\Vert ^{2}\left[\frac{\left\langle H_{U}\tau,\tau\right\rangle }{2}-K\left\Vert \tau\right\Vert ^{2}-L_{P}\left\Vert \tau\right\Vert \right]+\\
\left\Vert \eta\right\Vert ^{2}\left[\lambda_{f}+\frac{1}{2}\left\langle \nabla U,f\right\rangle -\frac{M_{PU}}{2}\left\Vert \tau\right\Vert \right]\ge\\
\left[\left\Vert \zeta\right\Vert -\frac{1}{2}\left\Vert \eta\right\Vert \left|\left\langle \nabla U,\tau\right\rangle \right|\right]^{2}+\frac{\left\Vert \eta\right\Vert ^{2}\left\Vert \tau\right\Vert ^{2}}{2}\left[\mu_{U}-2K\right]+\\
+\left\Vert \eta\right\Vert ^{2}\left[\lambda_{f}+\frac{\left\langle \nabla U,f\right\rangle }{2}\right]_{x=x(t)}-M_{P}\left\Vert \zeta\right\Vert \left\Vert \eta\right\Vert -\left(L_{P}+\frac{M_{PU}}{2}\right)Z\left\Vert \eta\right\Vert ^{2},
\end{gather*}
 Since there holds the inequality~(\ref{eq:lambdaf>sigmaZ}) and
$\left|\left\langle \nabla U,\tau\right\rangle \right|\le M_{U}Z$,
then there exists a constant $\alpha_{1}>0$ such that
\begin{gather*}
\dot{Q}(\eta,\zeta;t)\ge\alpha_{1}\left(\left\Vert \zeta\right\Vert ^{2}+\left\Vert \eta\right\Vert ^{2}\right).
\end{gather*}
 Let $\Theta_{s}^{t}:T_{x_{\ast}(s)}\mathcal{M}\mapsto T_{x_{\ast}(t)}\mathcal{M}$
be the cocycle of parallel shift along the mapping $x_{\ast}(\cdot)$
from point $x_{\ast}(s)$ to point $x_{\ast}(t)$. For any $\eta\in T_{x_{\ast}(s)}$$\mathcal{M}$,
there holds
\begin{gather*}
\nabla_{\tau(t)}\Theta_{s}^{t}\eta=0\quad\forall t\in\mathbb{R},\quad\Theta_{s}^{s}=\mathrm{Id,}\quad\Theta_{s}^{t}\Theta_{r}^{s}=\Theta_{r}^{t}.
\end{gather*}
After the change of variables
\begin{gather*}
\eta=\Theta_{0}^{t}y,\quad\zeta=\Theta_{0}^{t}z,\quad y,z\in T_{x_{\ast}(0)}\mathcal{M}
\end{gather*}
the system in variations takes the form
\begin{gather}
\dot{y}=z,\quad\dot{z}=A(t)y\label{eq:reduced-var-sys}
\end{gather}
where
\begin{gather*}
A(t):=\Theta_{t}^{0}\left[\nabla f(t\omega,x)\Theta_{0}^{t}y-R(\Theta_{0}^{t}y,\tau)\tau\right]_{x=x_{\ast}(t),\tau=\tau_{\ast}(t)}+\\
\left[\nabla P\left(t\omega,x\right)(\Theta_{0}^{t}y,\tau)+P\left(t\omega,x\right)\Theta_{0}^{t}z\right]_{x=x_{\ast}(t),\tau=\tau_{\ast}(t)},
\end{gather*}
$\tau_{\ast}(t):=\dot{x}_{\ast}(t)$. Since $\left\langle \Theta_{0}^{t}y,\Theta_{0}^{t}z\right\rangle =\left\langle y,z\right\rangle $,
it follows from the above that the derivative of quadratic form
\begin{gather*}
Q_{0}(y,z;t):=\left\langle y,z\right\rangle +\frac{\left\Vert y\right\Vert ^{2}}{2}\left\langle \nabla U(x_{\ast}(t)),\tau(t)\right\rangle
\end{gather*}
along solutions of System~(\ref{eq:reduced-var-sys}) is positive
definite. It is known that the existence of non-degenerate quadratic
form with the above property implies that System~(\ref{eq:reduced-var-sys})
is exponentially dichotomic~\cite{Sam2002}. Q.E.D.

The proof of Theorem~\ref{thm:hyperbol} is obviously follows from
the above one by letting $P(\varphi,x)=0$.

\subsection*{Proofs of Theorems~\ref{thm:quasiper-sol-f} and \ref{thm:quasiper-sol-f-P}}

Let us proceed to the proof of Theorem~\ref{thm:quasiper-sol-f-P}.
Theorem~\ref{thm:quasiper-sol-f} will immediately follow from Theorem~\ref{thm:quasiper-sol-f-P}.
By Theorem~\ref{thm:Th-B-sol-f-P} the domain $\mathcal{D}$ contains
a solution $x_{\ast}(\cdot):\mathbb{R}\mapsto\mathcal{D}$ of System~(\ref{eq:def-mu_U})
such that $\sup_{t\in\mathbb{R}}\left\Vert \dot{x}_{\ast}(t)\right\Vert \le z_{\ast}$.
Let us show that the solution with the above properties is unique.
Suppose that there exist two solutions $x_{i}(\cdot):\mathbb{R}\mapsto\mathcal{D}$
of System~(\ref{eq:qp-sys-f-P}) such that $\sup_{t\in\mathbb{R}}\left\Vert \dot{x}_{i}(t)\right\Vert \le z_{\ast}$,
$i\in\left\{ 1,2\right\} $. By Propositions~\ref{prop:nondegen}
and \ref{prop:connect-map}, with the help of implicit function theorem
and a continuation procedure one can construct a smooth vector field
$\xi(\cdot):\mathbb{R}\mapsto T\mathcal{M}$ along $x_{1}(\cdot)$
such that
\begin{gather*}
\xi(t)\in T_{x_{1}(t)}\mathcal{M},\quad\left\Vert \xi(t)\right\Vert \le d,\quad\frac{\partial}{\partial s}\Bigl|_{s=0}x(s,\xi(t))=\xi(t),\\
x(0,\xi(t))=x_{1}(t),\quad x(1,\xi(t))=x_{2}(t).
\end{gather*}

Introduce the smooth mapping $\chi(\cdot,\cdot):[0,1]\times\mathbb{R}\mapsto\mathcal{D}$
by the equality $\chi(s,t):=x(s,\xi(t))$ and define the tangent vector
fields $\chi^{\prime}(\cdot,\cdot)$, $\dot{\chi}(\cdot,\cdot)$ along
this mapping as
\begin{gather*}
\chi^{\prime}(s,t):=\frac{\partial}{\partial s}x(s,\xi(t)),\quad\dot{\chi}(s,t):=\frac{\partial}{\partial t}x(s,\xi(t)).
\end{gather*}
 Define also the function
\begin{gather}
S(t):=\left\langle \chi^{\prime},\dot{\chi}\right\rangle \Bigl|_{s=0}^{s=1}\equiv\left\langle \chi^{\prime}(1,t),\dot{x}_{2}(t)\right\rangle -\left\langle \xi(t),\dot{x}_{1}(t)\right\rangle \label{eq:def-S(t)}
\end{gather}
 and calculate its derivative:
\begin{gather*}
\dot{S}(t)=\left[\frac{\mathrm{d}}{\mathrm{d}t}\left\langle \chi^{\prime},\dot{\chi}\right\rangle \right]\Bigl|_{s=0}^{s=1}=\left[\left\langle \nabla_{\dot{\chi}}\chi^{\prime},\dot{\chi}\right\rangle +\left\langle \chi^{\prime},\nabla_{\dot{\chi}}\dot{\chi}\right\rangle \right]\Bigl|_{s=0}^{s=1}=\\
\left[\left\langle \nabla_{\chi^{\prime}}\dot{\chi},\dot{\chi}\right\rangle +\left\langle \chi^{\prime},\nabla_{\dot{\chi}}\dot{\chi}\right\rangle \right]\Bigl|_{s=0}^{s=1}=\\
\left[\frac{\partial}{\partial s}\frac{\left\Vert \dot{\chi}\right\Vert ^{2}}{2}\right]_{s=0}^{s=1}+\left\langle \chi^{\prime},f(t\omega,\chi)+P(t\omega,\chi)\dot{\chi}\right\rangle \Bigl|_{s=0}^{s=1}=\\
\intop_{0}^{1}\left[\frac{\partial^{2}}{\partial s^{2}}\frac{\left\Vert \dot{\chi}\right\Vert ^{2}}{2}+\frac{\partial}{\partial s}\left\langle \chi^{\prime},f(t\omega,\chi)+P(t\omega,\chi)\dot{\chi}\right\rangle \right]\mathrm{d}s.
\end{gather*}
Using the equalities
\begin{gather*}
\nabla_{\dot{\chi}}\chi^{\prime}=\nabla_{\chi^{\prime}}\dot{\chi},\quad\nabla_{\chi^{\prime}}^{2}\dot{\chi}=\nabla_{\chi^{\prime}}\nabla_{\dot{\chi}}\chi^{\prime}=\nabla_{\dot{\chi}}\nabla_{\chi^{\prime}}\chi^{\prime}-R(\dot{\chi},\chi^{\prime})\chi^{\prime},\\
\nabla_{\chi^{\prime}}\chi^{\prime}=\left\Vert \chi^{\prime}\right\Vert ^{2}\nabla U(\chi)/2,
\end{gather*}
we obtain
\begin{gather*}
\frac{\partial^{2}}{\partial s^{2}}\frac{\left\Vert \dot{\chi}\right\Vert ^{2}}{2}=\left\Vert \nabla_{\chi^{\prime}}\dot{\chi}\right\Vert ^{2}+\left\langle \nabla_{\dot{\chi}}\nabla_{\chi^{\prime}}\chi^{\prime},\dot{\chi}\right\rangle -\left\langle R(\dot{\chi},\chi^{\prime})\chi^{\prime},\dot{\chi}\right\rangle =\\
\left\Vert \nabla_{\dot{\chi}}\chi^{\prime}\right\Vert ^{2}+\left\langle \nabla_{\dot{\chi}}\chi^{\prime},\chi^{\prime}\right\rangle \left\langle \nabla U(\chi),\dot{\chi}\right\rangle +\frac{\left\Vert \chi^{\prime}\right\Vert ^{2}}{2}\left\langle H_{U}(\chi)\dot{\chi},\dot{\chi}\right\rangle -\\
K(\chi)\left[\left\Vert \chi^{\prime}\right\Vert ^{2}\left\Vert \dot{\chi}\right\Vert ^{2}-\left\langle \chi^{\prime},\dot{\chi}\right\rangle ^{2}\right]\ge\\
\left\Vert \nabla_{\dot{\chi}}\chi^{\prime}\right\Vert ^{2}-\left\Vert \nabla_{\dot{\chi}}\chi^{\prime}\right\Vert \left\Vert \chi^{\prime}\right\Vert \left|\left\langle \nabla U(\chi),\dot{\chi}\right\rangle \right|+\left\Vert \chi^{\prime}\right\Vert ^{2}\left(\frac{1}{2}\left\langle H_{U}(\chi)\dot{\chi},\dot{\chi}\right\rangle -K(\chi)\left\Vert \dot{\chi}\right\Vert ^{2}\right)=\\
\left[\left\Vert \nabla_{\dot{\chi}}\chi^{\prime}\right\Vert -\frac{\left|\left\langle \nabla U(\chi),\dot{\chi}\right\rangle \right|}{2}\right]^{2}+\frac{\left\Vert \chi^{\prime}\right\Vert ^{2}}{2}\left[\mu_{U}(\chi)-2K(\chi)\right]\ge0;
\end{gather*}
\begin{gather*}
\frac{\partial}{\partial s}\left\langle f(\varphi,\chi),\chi^{\prime}\right\rangle =\left\langle \nabla f(\varphi,\chi)\chi^{\prime},\chi^{\prime}\right\rangle +\frac{\left\Vert \chi^{\prime}\right\Vert ^{2}}{2}\left\langle f(\varphi,\chi),\nabla U(\chi)\right\rangle \ge\\
\left\Vert \chi^{\prime}\right\Vert ^{2}\left[\lambda_{f}(\varphi,\chi)+\frac{1}{2}\left\langle f(\varphi,\chi),\nabla U(\chi)\right\rangle \right].
\end{gather*}
Since $\frac{\partial^{2}}{\partial s^{2}}\left\Vert \dot{\chi}(s,t)\right\Vert ^{2}\ge0$,
then $ $$ $
\begin{gather*}
\left\Vert \dot{\chi}(s,t)\right\Vert ^{2}\le\frac{1}{2}\left[\left\Vert \dot{\chi}(0,t)\right\Vert ^{2}+\left\Vert \dot{\chi}(1,t)\right\Vert ^{2}\right]=\frac{1}{2}\left[\left\Vert \dot{x}_{1}(t)\right\Vert ^{2}+\left\Vert \dot{x}_{2}(t)\right\Vert ^{2}\right]\le z_{\ast}^{2}.
\end{gather*}
Now
\begin{gather*}
\left|\frac{\partial}{\partial s}\left\langle P(\varphi,x)\dot{\chi},\chi^{\prime}\right\rangle \right|\le\left|\left\langle \nabla_{\chi^{\prime}}\left[P(\varphi,\chi)\dot{\chi}\right],\chi^{\prime}\right\rangle \right|+\left|\left\langle P(\varphi,\chi)\dot{\chi},\nabla_{\chi^{\prime}}\chi^{\prime}\right\rangle \right|\le\\
\left|\left\langle \nabla P(\varphi,\chi)(\chi^{\prime},\dot{\chi}),\chi^{\prime}\right\rangle \right|+\left|\left\langle P(\varphi,\chi)\nabla_{\chi^{\prime}}\dot{\chi},\chi^{\prime}\right\rangle \right|+\left|\left\langle P(\varphi,\chi)\dot{\chi},\nabla_{\chi^{\prime}}\chi^{\prime}\right\rangle \right|\le\\
L_{P}(\varphi,\chi)z_{\ast}\left\Vert \chi^{\prime}\right\Vert ^{2}+M_{P}(\varphi,\chi)\left\Vert \nabla_{\dot{\chi}}\chi^{\prime}\right\Vert \left\Vert \chi^{\prime}\right\Vert +\frac{M_{PU}(\varphi,\chi)z_{\ast}}{2}\left\Vert \chi^{\prime}\right\Vert ^{2}.
\end{gather*}
Just like in the proof of Theorem~(\ref{thm:hyperbol-P}), it is
not hard to show that the above inequalities together with condition~(\ref{eq:labdaf-g-sigma})
imply that there exists $\alpha_{2}>0$ such that there holds the
inequality
\begin{gather*}
\dot{S}(t)\ge\alpha_{2}\intop_{0}^{1}\left(\left\Vert \nabla_{\dot{\chi}}\chi^{\prime}\right\Vert ^{2}+\left\Vert \chi^{\prime}\right\Vert ^{2}\right)\mathrm{d}s.
\end{gather*}
 By Proposition~\ref{prop:connect-map} $\sup_{t\in\mathbb{R}}\left\Vert \xi(t)\right\Vert \le d$,
and on account of~(\ref{eq:cons-low}) and~(\ref{eq:U^*-U_*}) we
have
\begin{equation}
\left\Vert \xi(t)\right\Vert ^{2}\mathrm{e}^{U_{\ast}-U^{\ast}}\le\left\Vert \chi^{\prime}(s,t)\right\Vert ^{2}=\left\Vert \xi(t)\right\Vert ^{2}\exp\left(U(x_{1}(t))-U(\chi(s,t))\right)\le\left\Vert \xi(t)\right\Vert ^{2}\mathrm{e}^{U^{\ast}-U_{\ast}}.\label{eq:chi<xi}
\end{equation}
 Thus we obtain the inequality $\dot{S}(t)\ge\alpha_{2}\mathrm{e}^{U_{\ast}-U^{\ast}}\left\Vert \xi(t)\right\Vert ^{2}$
which together with boundedness of $\left\Vert \xi(t)\right\Vert $
and $\left|S(t)\right|$ yields that
\[
-\infty<S(-\infty)<S(0)<S(+\infty)<\infty.
\]
 Now it turns out that either $l_{+}:=\liminf_{t\to+\infty}\left\Vert \xi(t)\right\Vert >0$
or $l_{-}:=\liminf_{t\to-\infty}\left\Vert \xi(t)\right\Vert >0$.
In fact, if $l_{-}=l_{+}=0$ then~(\ref{eq:def-S(t)}) and(\ref{eq:chi<xi})
implies that $S(-\infty)=S(+\infty)=0$ and we reach the contradiction.
But if $l_{+}>0$, then $S(+\infty)=+\infty$, and if $l_{-}>0$ then
$S(-\infty)=-\infty$. Both these cases produce the contradiction.
Hence, we have proved the announced uniqueness.

Obviously, the above reasoning is valid also for any system of the
form
\[
\nabla_{\dot{x}}\dot{x}=f(t\omega+\varphi,x)+P(t\omega+\varphi,x)\dot{x}\quad\forall\varphi\in\mathbb{T}^{k}.
\]
 Hence, for any $\varphi\in\mathbb{T}^{k}$ there exists the solution
$x_{\ast}(\cdot,\varphi):\mathbb{R}\mapsto\mathcal{D}$ which generates
the single valued mapping $x_{\ast}(\cdot,\cdot):\mathbb{R}\times\mathbb{T}^{k}\mapsto\mathcal{D}$
such that $\sup_{(t,\varphi)\in\mathbb{R}\times\mathbb{T}^{k}}\left\Vert \dot{x}_{\ast}(t,\varphi)\right\Vert \le z_{\ast}$.
But then
\begin{gather*}
x_{\ast}(t+s,\varphi)=x_{\ast}(s,t\omega+\varphi)\quad\forall\{t,s\}\subset\mathbb{R},\;\varphi\in\mathbb{T}^{k}.
\end{gather*}
 If we put here $s=0$ and define the mapping $h(\cdot):=x_{\ast}(0,\cdot):\mathbb{T}^{k}\mapsto\mathcal{D}$
then we obtain
\begin{gather*}
x_{\ast}(t,\varphi)=h(t\omega+\varphi).
\end{gather*}
To show that $x_{\ast}(t,\varphi)$ is quasiperiodic, let us prove
that $h(\cdot)$ is continuous. Suppose that the opposite is true.
Then there exists $\varphi_{\ast}\in\mathbb{T}^{k}$ and a sequence
$\left\{ \varphi_{i}\right\} \subset\mathbb{T}^{k}$ converging to
$\varphi_{\ast}$ such that
\begin{gather*}
h(\varphi_{i})\to\tilde{x}_{0}\in\mathrm{cl}(\mathcal{D}),\quad\dot{x}_{\ast}(0,\varphi_{i})\to\tilde{\xi}_{0},\quad\left\Vert \tilde{\xi}_{0}\right\Vert \le z_{\ast},\quad,\tilde{x}_{0}\ne x_{\ast}(0,\varphi_{\ast}).
\end{gather*}
Consider the non-extendable solution $\tilde{x}(\cdot):I\mapsto\mathcal{M}$
of the initial-value problem
\begin{gather*}
\nabla_{\dot{x}}\dot{x}=f(t\omega+\varphi_{\ast},x)+P(t\omega+\varphi_{\ast},x)\dot{x},\quad x(0)=\tilde{x}_{0},\quad\dot{x}(0)=\tilde{\xi}_{0}.
\end{gather*}
 Since each system $\nabla_{\dot{x}}\dot{x}=f(t\omega+\varphi_{i},x)+P(t\omega+\varphi_{i},x)\dot{x}$
is equivalent to the first order system
\begin{gather*}
\dot{x}=\eta,\quad\nabla_{\dot{x}}\eta=f(t\omega+\varphi_{i},x)+P(t\omega+\varphi_{i},x)\dot{x}
\end{gather*}
and $\left\{ f(\varphi+\varphi_{i},x)+P(\varphi+\varphi_{i},x)\dot{x}\right\} $
converges to $f(\varphi+\varphi_{\ast},x)+P(\varphi+\varphi_{i},x)\dot{x}$
uniformly with respect to $\varphi\in\mathbb{T}^{k}$, $x\in\mathrm{cl}(\mathcal{D})$,
$\dot{x}\in T_{x}\mathcal{M}$, and $\left\Vert \dot{x}\right\Vert \le z_{\ast}$,
then for any closed segment $J\subset I$ the sequence $\left\{ \left(x_{\ast}(t,\varphi_{i}),\dot{x}_{\ast}(t,\varphi_{i})\right)\right\} $
converges to $\left(\tilde{x}(t),\frac{\mathrm{d}}{\mathrm{d}t}\tilde{x}(t)\right)$
uniformly with respect to $t\in J$. This yields that $\tilde{x}(\cdot):I\mapsto\mathrm{cl}(\mathcal{D})$,
$\sup_{t\in I}\left\Vert \frac{\mathrm{d}}{\mathrm{d}t}\tilde{x}(t)\right\Vert \le z_{\ast}$
and hence $I=\mathbb{R}$. The same arguments as in the proof of Theorem~\ref{thm:Th-B-sol-f-P}
allow us to show that $\tilde{x}(\cdot):\mathbb{R}\mapsto\mathcal{D}$.
Thus we reach the contradiction with uniqueness of solution taking
values in $\mathcal{D}$ and possessing the derivative of norm bounded
by $z_{\ast}$.

\section{Quasiperiodic motion of charged particle on unit sphere\label{sec:charged particle}}

Let $\mathbb{E}^{3}=\left(\mathbb{R}^{3},\left\langle \cdot,\cdot\right\rangle \right)$
be the 3-dimensional Euclidean space endowed with a scalar product
$\left\langle \cdot,\cdot\right\rangle $ and cross-product $\cdot\times\cdot$.
$ $ Consider a charged particle of unit mass which is constrained
to move on the surface of the sphere $\mathbb{S}^{2}:=\left\{ \mathbf{x}\in\mathbb{E}^{3}:\left\Vert \mathbf{x}\right\Vert ^{2}=1\right\} $
by the applied force $\Phi$ represented in the form
\begin{gather*}
\boldsymbol{\Phi}(t\omega,\mathbf{x},\dot{\mathbf{x}})=-\frac{\mathbf{x}-\mathbf{a}}{\left\Vert \mathbf{x}-\mathbf{a}\right\Vert ^{3}}+\mathbf{E}(t\omega)+\dot{\mathbf{x}}\times\mathbf{B}(t\omega)-\kappa\dot{\mathbf{x}}.
\end{gather*}
Here $\mathbf{a}\in\mathbb{E}^{3}$ is a vector of norm $a:=\left\Vert \mathbf{a}\right\Vert $;
$b$, $\kappa$ are positive parameters; $\mathbf{E}(\cdot):\mathbb{T}^{k}\mapsto\mathbb{E}^{3}$
and $\mathbf{B}(\cdot):\mathbb{T}^{k}\mapsto\mathbb{E}^{3}$ are smooth
mappings; $\omega\in\mathbb{R}^{k}$ is a frequency vector. The force
$\boldsymbol{\Phi}$ can $ $ be naturally interpreted as the superposition
of three forces: the Coulomb force caused by a charge placed at point
$\mathbf{a}$; the Lorentz force caused by the electric field $\mathbf{E}$
and the magnetic field $\mathbf{B}$; the damping force  $-\varkappa\dot{\mathbf{x}}$.

Subtracting from $\boldsymbol{\Phi}(t\omega,\mathbf{x},\dot{\mathbf{x}})$
its normal component and introducing unit vector $\mathbf{k}:=-\mathbf{a}/a$,
we find that in the case under consideration the forces affecting
the motion of the constrained particle are

\begin{gather*}
\mathbf{f}(t\omega,\mathbf{x})=-\frac{\mathbf{x}+a\mathbf{k}}{\left\Vert \mathbf{x}+a\mathbf{k}\right\Vert ^{3}}+\mathbf{E}(t\omega)+\left\langle \frac{\mathbf{x}+a\mathbf{k}}{\left\Vert \mathbf{x}+a\mathbf{k}\right\Vert ^{3}}-\mathbf{E}(t\omega),\mathbf{x}\right\rangle \mathbf{x},\\
P(t\omega,\mathbf{x})\dot{\mathbf{x}}=\dot{\mathbf{x}}\times\mathbf{B}(t\omega)-\left\langle \dot{\mathbf{x}}\times\mathbf{B}(t\omega),\mathbf{x}\right\rangle \mathbf{x}-\kappa\dot{\mathbf{x}}.
\end{gather*}

Recall that if $\mathbf{v}(\cdot):\mathbb{S}^{2}\mapsto\mathbb{E}^{3}$
is a smooth tangent vector field on $\mathbb{S}^{2}$, i.e. $\left\langle \mathbf{v}(\mathbf{x}),\mathbf{x}\right\rangle =0$
for any $\mathbf{x}\in\mathbb{S}^{2}$, then for any $\mathbf{h}\in T_{\mathbf{x}}\mathbb{S}^{2}$
we have
\[
\nabla_{\mathbf{h}}\mathbf{v}(\mathbf{x})=\mathbf{v}^{\prime}(\mathbf{x})\mathbf{h}-\left\langle \mathbf{v}^{\prime}(\mathbf{x})\mathbf{h},\mathbf{x}\right\rangle \mathbf{x}.
\]

First consider the case where the influence of magnetic field and
the damping force can be neglected.
\begin{thm}
Let $\mathbf{B}(\varphi)\equiv0$, $\kappa=0$. If there holds the
inequality
\begin{gather}
\frac{a}{\left(1+a\right)^{3}}-\left\langle \mathbf{E}(\varphi),\mathbf{k}\right\rangle >0\quad\forall\varphi\in\mathbb{T}^{k}\label{eq:ab>E}
\end{gather}
and there exists a point $\varphi_{0}\in\mathbb{T}^{k}$ such that
$\mathbf{E}(\varphi_{0})\not\parallel\mathbf{k}$, then the system
of charged particle on $\mathbb{S}^{1}$ has a unique $\omega$-quasiperiodic
solution located in the hemisphere $\left\{ \mathbf{x}\in\mathbb{S}^{1}:0<\left\langle \mathbf{x},\mathbf{k}\right\rangle \le1\right\} $.
This solution is hyperbolic. \label{thm:charged-particle}\end{thm}
\begin{proof}
Let a unit tangent vector $\mathbf{e}\in T_{\mathbf{x}}\mathbb{S}^{1}$
be taken at will. Then in view of $\left\langle \mathbf{x},\mathbf{e}\right\rangle =0$
we have
\begin{gather*}
\left\langle \nabla_{\mathbf{e}}\mathbf{f}(\varphi,\mathbf{x}),\mathbf{e}\right\rangle =\left\langle \mathbf{f}_{\mathbf{x}}^{\prime}(\varphi,\mathbf{x})\mathbf{e},\mathbf{e}\right\rangle =\\
-\frac{1}{\left\Vert \mathbf{x}+a\mathbf{k}\right\Vert ^{3}}+\frac{3a^{2}\left\langle \mathbf{k},\mathbf{e}\right\rangle ^{2}}{\left\Vert \mathbf{x}+a\mathbf{k}\right\Vert ^{5}}+\left\langle \frac{\mathbf{x}+a\mathbf{k}}{\left\Vert \mathbf{x}+a\mathbf{k}\right\Vert ^{3}}-\mathbf{E}(\varphi),\mathbf{x}\right\rangle =\\
\frac{a\left\langle \mathbf{k},\mathbf{x}\right\rangle }{\left\Vert \mathbf{x}+a\mathbf{k}\right\Vert ^{3}}+\frac{3a^{2}\left\langle \mathbf{k},\mathbf{e}\right\rangle ^{2}}{\left\Vert \mathbf{x}+a\mathbf{k}\right\Vert ^{5}}-\left\langle \mathbf{E}(\varphi),\mathbf{x}\right\rangle .
\end{gather*}
It is not hard to see that
\begin{gather*}
\lambda_{\mathbf{f}}(\varphi,\mathbf{x})=\frac{a\left\langle \mathbf{k},\mathbf{x}\right\rangle }{\left\Vert \mathbf{x}+a\mathbf{k}\right\Vert ^{3}}-\left\langle \mathbf{E}(\varphi),\mathbf{x}\right\rangle .
\end{gather*}
In particular,
\begin{gather*}
\lambda_{\mathbf{f}}(\varphi,\mathbf{k})=\frac{a}{\left(1+a\right){}^{3}}-\left\langle \mathbf{E}(\varphi),\mathbf{k}\right\rangle >0.
\end{gather*}
Now, in order to apply Theorem~\ref{thm:quasiper-sol-f}, we are
going to find the appropriate domain $\mathcal{D}$ and function $U(\cdot)$.
$ $ Observe that for a function $U(\cdot):\mathbb{S}^{2}\mapsto\mathbb{R}$
such that $\nabla U(\mathbf{k})=0$, the inequality(\ref{eq:lamb_f+nablaU})
holds true at least near $\mathbf{k}$. For this reason, we define
the domain
\begin{gather*}
\mathcal{D}:=\left\{ \mathbf{x}\in\mathbb{S}^{2}:\;\rho<\left\langle \mathbf{k},\mathbf{x}\right\rangle \le1\right\}
\end{gather*}
where $\rho\in(0,1)$ will be determined later.

Set $u(\mathbf{x}):=-\left\langle \mathbf{k},\mathbf{x}\right\rangle $.
To satisfy the conditions of Theorem~\ref{thm:quasiper-sol-f} we
seek the function $U(\cdot)$ in the form $U(\mathbf{x})=y\circ u(\mathbf{x})$.
For $\mathbf{e}\in T_{\mathbf{x}}\mathbb{S}^{2}$, $\left\Vert \mathbf{e}\right\Vert =1$,
we have
\begin{gather*}
\nabla_{\mathbf{e}}\nabla U(\mathbf{x})=\left[y^{\prime\prime}(u)\left\langle \nabla u,\mathbf{e}\right\rangle \nabla u+y^{\prime}(u)\nabla_{\mathbf{e}}\nabla u\right]_{u=u(\mathbf{x})}.
\end{gather*}
 On account that
\begin{gather}
\nabla u(\mathbf{x})=-\mathbf{k}+\left\langle \mathbf{k},\mathbf{x}\right\rangle \mathbf{x}=-u(\mathbf{x})\mathbf{x}-\mathbf{k},\quad\nabla_{\mathbf{e}}\nabla u(\mathbf{x})=-u(\mathbf{x})\mathbf{e},\label{eq:nabla-u-Hess-u}
\end{gather}
 we obtain
\begin{gather*}
\left\langle H_{U}(\mathbf{x})\mathbf{e},\mathbf{e}\right\rangle =\left[y^{\prime\prime}(u)\left\langle \mathbf{k},\mathbf{e}\right\rangle ^{2}+y^{\prime}(u)\left\langle \mathbf{k},\mathbf{x}\right\rangle \right]_{u=u(\mathbf{x})}=\left[\left\langle \mathbf{k},\mathbf{e}\right\rangle ^{2}y^{\prime\prime}(u)-uy^{\prime}(u)\right]_{u=u(\mathbf{x})},\\
\left\langle \nabla U(\mathbf{x}),\mathbf{e}\right\rangle ^{2}=y^{\prime}{}^{2}(u(\mathbf{x}))\left\langle \nabla u(\mathbf{x}),\mathbf{e}\right\rangle ^{2}=\left\langle \mathbf{k},\mathbf{e}\right\rangle ^{2}y^{\prime}{}^{2}(u(\mathbf{x})).
\end{gather*}
As is well known, $K(\mathbf{x})=1$ for $\mathcal{M}=\mathbb{S}^{2}$,
and in our case the inequality $\mu_{U}(x)\ge2K(x)$ of Hypothesis~\textbf{H3}
takes the form
\begin{gather*}
\left[\left\langle \mathbf{k},\mathbf{e}\right\rangle ^{2}(y^{\prime\prime}(u)-y^{\prime2}(u)/2)-uy^{\prime}(u)\right]_{u=u(\mathbf{x})}\ge2\quad\forall u\in[-1,-\rho).
\end{gather*}
This inequality obviously turns into equality if we put $y=-\ln u^{2}$.
Hence, it is naturally to define
\begin{gather*}
U(\mathbf{x}):=-\ln u^{2}(\mathbf{x}).
\end{gather*}
$\frac{}{}$Let us verify Hypothesis \textbf{H1}. Under the above
choice of $U(\cdot)$, we get
\begin{gather}
\nabla U(\mathbf{x})=-\frac{2\nabla u(\mathbf{x})}{u(\mathbf{x})},\quad\left\langle H_{U}(x)\mathbf{e},\mathbf{e}\right\rangle =\frac{2\left\langle \mathbf{k},\mathbf{e}\right\rangle ^{2}}{u^{2}(\mathbf{x})}+2\quad\Rightarrow\quad\lambda_{U}(\mathbf{x})=2.\label{eq:nabla-U-Hess-U}
\end{gather}
Since the third addendum in expression for $\mathbf{f}(\varphi,\mathbf{x})$
is orthogonal to $ $$\mathbb{S}^{2}$, then on account of~(\ref{eq:nabla-u-Hess-u})
\begin{gather}
\left\langle \nabla U(\mathbf{x}),\mathbf{f}(\varphi,\mathbf{x})\right\rangle =\frac{2}{u(\mathbf{x})}\left\langle \nabla u(\mathbf{x}),\frac{\mathbf{x}+a\mathbf{k}}{\left\Vert \mathbf{x}+a\mathbf{k}\right\Vert ^{3}}-\mathbf{E}(\varphi)\right\rangle =\nonumber \\
-\frac{2}{u(\mathbf{x})}\left[\frac{a\left(1-u^{2}(\mathbf{x})\right)}{\left\Vert \mathbf{x}+a\mathbf{k}\right\Vert ^{3}}-\left\langle \mathbf{E}(\varphi),\mathbf{k}+u(\mathbf{x})\mathbf{x}\right\rangle \right].\label{eq:nblaU-f}
\end{gather}
We have to show that this function has negative minimum in $\mathbb{T}^{k}\times\mathrm{cl}(\mathcal{D})$,
or, what is the same, the parameter $q$ (see~(\ref{eq:def-q1}))
is correctly defined, i.e. actually is positive. Let $\mathbf{i},\mathbf{j},\mathbf{k}$
be the standard right-oriented orthonormal basis in $\mathbb{E}^{3}$.
Denote by $E_{\mathbf{i}}(\varphi),E_{\mathbf{j}}(\varphi),E_{\mathbf{k}}(\varphi)$
the projections of $\mathbf{E}(\varphi)$ onto $\mathbf{i},\mathbf{j},\mathbf{k}$
respectively. It turns out that for fixed $\varphi\in\mathbb{T}^{k}$
and $s\in[\rho,1]$ the conditional maximum
\begin{gather*}
M(s,\varphi):=\max\left\{ -\frac{\left\langle \nabla U(x),\mathbf{f}(\varphi,\mathbf{x})\right\rangle }{\lambda_{U}(\mathbf{x})}:\left\langle \mathbf{k},\mathbf{x}\right\rangle =s\right\}
\end{gather*}
is attained at point $\mathbf{x}\in\mathbb{S}^{2}$ such that
\begin{gather*}
\left\langle \mathbf{i},\mathbf{x}\right\rangle =-\frac{\sqrt{1-s^{2}}E_{\mathbf{i}}(\varphi)}{\sqrt{E_{\mathbf{i}}^{2}(\varphi)+E_{\mathbf{j}}^{2}(\varphi)}},\quad\left\langle \mathbf{j},\mathbf{x}\right\rangle =-\frac{\sqrt{1-s^{2}}E_{\mathbf{j}}(\varphi)}{\sqrt{E_{\mathbf{i}}^{2}(\varphi)+E_{\mathbf{j}}^{2}(\varphi)}},\quad\left\langle \mathbf{k},\mathbf{x}\right\rangle =s.
\end{gather*}
Hence,
\begin{gather*}
M(s,\varphi)=-\frac{a(1-s^{2})}{s(1+2sa+a^{2})^{3/2}}+\frac{(1-s^{2})E_{\mathbf{k}}(\varphi)}{s}+\sqrt{1-s^{2}}\sqrt{E_{\mathbf{i}}^{2}(\varphi)+E_{\mathbf{j}}^{2}(\varphi)}.
\end{gather*}
$ $Since there exists a point $\varphi_{0}\in\mathbb{T}^{k}$ such
that $\mathbf{E}(\varphi_{0})\not\parallel\mathbf{k}$, then $M(s,\varphi_{0})>0$
if $s$ is sufficiently close to 1. Hence we show that
\begin{gather*}
q^{2}=\max\left\{ M(s,\varphi):\varphi\in\mathbb{T}^{k},s\in[\rho,1]\right\} >0
\end{gather*}
 and that Hypothesis \textbf{H1} is satisfied.$ $

Let us verify Hypothesis \textbf{H2.} Since $\partial\mathcal{D}:=\left\{ \mathbf{x}\in\mathbb{S}^{2}:\; u(\mathbf{x})=-\rho\right\} $,
then on account of~(\ref{eq:nabla-u-Hess-u}) the outward unit normal
at $\mathbf{x}\in\partial\mathcal{D}$ is

\begin{gather*}
\mathbf{n}(\mathbf{x})=\frac{\nabla u(\mathbf{x})}{\left\Vert \nabla u(\mathbf{x})\right\Vert }=\frac{\nabla u(\mathbf{x})}{\sqrt{1-u^{2}(\mathbf{x})}}=\frac{\rho}{2\sqrt{1-\rho^{2}}}\nabla U(\mathbf{x}),\quad\mathbf{x}\in\partial\mathcal{D}.
\end{gather*}
 Now from~(\ref{eq:nabla-u-Hess-u}), (\ref{eq:nabla-U-Hess-U})
it follows that
\begin{gather*}
\lambda_{II}(\mathbf{x})=\frac{\left\langle \nabla_{\mathbf{e}}\nabla u(\mathbf{x}),\mathbf{e}\right\rangle }{\left\Vert \nabla u(\mathbf{x})\right\Vert }=\frac{\rho}{\sqrt{1-\rho^{2}}}>0\quad\forall\mathbf{x}\in\partial\mathcal{D}.
\end{gather*}
 Observe that $\mathbf{n}(\mathbf{x})=\frac{\rho}{2\sqrt{1-\rho^{2}}}\nabla U(\mathbf{x}).$
Now (\ref{eq:nblaU-f}) yields
\begin{gather*}
\min\left\{ \left\langle \mathbf{n}(\mathbf{x}),\mathbf{f}(\varphi,\mathbf{x})\right\rangle :\mathbf{x}\in\partial\mathcal{D}\right\} =-\frac{\rho}{\sqrt{1-\rho^{2}}}M(\rho,\varphi)=\\
\sqrt{1-\rho^{2}}\left[\frac{a}{(1+2\rho a+a^{2})^{3/2}}-E_{\mathbf{k}}(\varphi)\right]-\rho\sqrt{E_{\mathbf{i}}^{2}(\varphi)+E_{\mathbf{j}}^{2}(\varphi)},
\end{gather*}
\begin{gather*}
\lambda_{II}(\mathbf{x})\left\langle \mathbf{n}(\mathbf{x}),\mathbf{f}(\varphi,\mathbf{x})\right\rangle \ge\frac{\rho}{2}\left[\frac{a}{(1+2\rho a+a^{2})^{3/2}}-E_{\mathbf{k}}(\varphi)\right]-\frac{\rho^{2}}{2}\sqrt{\frac{E_{\mathbf{i}}^{2}(\varphi)+E_{\mathbf{j}}^{2}(\varphi)}{1-\rho^{2}}}.
\end{gather*}
The condition~(\ref{eq:ab>E}) implies that here the right-hand sides
of both inequalities are positive once $\rho$ is sufficiently small.
Besides,
\begin{gather*}
\left\langle \nabla U(\mathbf{x}),\mathbf{n}(\mathbf{x})\right\rangle =-\frac{2\left\Vert \nabla u(\mathbf{x})\right\Vert ^{2}}{u(\mathbf{x})\sqrt{1-u^{2}(\mathbf{x})}}=\frac{2\sqrt{1-\rho^{2}}}{\rho}>0\quad\forall\mathbf{x}\in\partial\mathcal{D}.
\end{gather*}
Thus we see that both Hypothesis \textbf{H2} and the inequality~(\ref{eq:II-cond})
holds true. .

To verify that the $U$-monotonicity condition is satisfied, observe
that

\begin{gather*}
\lambda_{\mathbf{f}}(\varphi,\mathbf{x})+\frac{1}{2}\left\langle \nabla U(\mathbf{x}),\mathbf{f}(\varphi,\mathbf{x})\right\rangle =\\
-\frac{1}{u(\mathbf{x})}\left[\frac{a}{\left\Vert \mathbf{x}+a\mathbf{k}\right\Vert ^{3}}-\left\langle \mathbf{E}(\varphi),\mathbf{k}\right\rangle \right]=-\frac{1}{u(\mathbf{x})}\left[\frac{a}{\left(1-2au(\mathbf{x})+a^{2}\right)^{3/2}}-\left\langle \mathbf{E}(\varphi),\mathbf{k}\right\rangle \right]\ge\\
\frac{a}{\left(1+a\right)^{3}}-\left\langle \mathbf{E}(\varphi),\mathbf{k}\right\rangle >0.
\end{gather*}
Since it has been have already shown that~(\ref{eq:cond-mu}) is
fulfilled, we complete the proof by applying Theorem~(\ref{thm:quasiper-sol-f})
\end{proof}
Now let us proceed to the  system perturbed by magnetic field and
damping. Define
\begin{gather*}
E:=\max_{\varphi\in\mathbb{T}^{k}}\left\Vert \mathbf{E}(\varphi)\right\Vert ,\quad\beta:=\max_{\varphi\in\mathbb{T}^{k}}\left\Vert \mathbf{B}(\varphi)\right\Vert /\rho,\quad\varkappa:=\kappa/\rho.
\end{gather*}

In order to make the results concerning the perturbed system more
demonstrative, consider the case where $E_{\mathbf{k}}(\varphi)=0$,
$\mathbf{B}(\varphi)\perp\mathbf{E}(\varphi)$ , i.e. $\mathbf{B}(\varphi)=B(\varphi)\mathbf{k}$
where $B(\cdot):\mathbb{T}^{k}\mapsto\mathbb{R}$ is a smooth function
(the typical case of crossed electric and magnetic field), and $\varkappa\le\beta$.
It is not hard to show that when applying the results of Section~\ref{sec:Not-hyp-res}
one can set
\begin{gather*}
M_{f}(\mathbf{x}):=\left\Vert \mathbf{x}+a\mathbf{k}\right\Vert ^{-2}+E,\\
C_{f}:=\frac{1}{2}\left[\max\left\{ \left(1+2as+a^{2}\right)^{-1}:\rho\le s\le1\right\} +E\right]=\\
\frac{1}{2}\left[E+(1+2a\rho+a^{2})^{-1}\right]\le\frac{1}{2}\left[E+(1+a^{2})^{-1}\right],
\end{gather*}

\begin{gather*}
M_{U}(\mathbf{x}):=\frac{2\sqrt{1-\left\langle \mathbf{k},\mathbf{x}\right\rangle ^{2}}}{\left\langle \mathbf{k},\mathbf{x}\right\rangle },\quad C_{U}:=\max_{\rho\le s\le1}\frac{2\sqrt{1-s^{2}}}{s}=\frac{2\sqrt{1-\rho^{2}}}{\rho},\\
M_{P}(\varphi,\mathbf{x}):=\rho\sqrt{\beta^{2}\left\langle \mathbf{k},\mathbf{x}\right\rangle ^{2}+\varkappa^{2}},\quad M_{PU}(\varphi,\mathbf{x}):=\frac{2\rho\sqrt{\left(1-\left\langle \mathbf{k},\mathbf{x}\right\rangle ^{2}\right)\left(\beta^{2}\left\langle \mathbf{k},\mathbf{x}\right\rangle ^{2}+\varkappa^{2}\right)}}{\left\langle \mathbf{k},\mathbf{x}\right\rangle },
\end{gather*}

\begin{gather*}
p:=\max_{0\le s\le1}\sqrt{\left(1-s^{2}\right)(\beta^{2}s^{2}+\varkappa^{2})}=\frac{\left(\beta^{2}+\varkappa^{2}\right)}{2\beta},\quad l:=\frac{(1+a)^{2}\varkappa\rho}{(1+a)^{2}E+1}.
\end{gather*}
(Here we use the equalities $\left\Vert \mathbf{e}\times\mathbf{k}-\left\langle \mathbf{e}\times\mathbf{k},\mathbf{x}\right\rangle \mathbf{x}\right\Vert ^{2}=\left\Vert \mathbf{e}\times\mathbf{k}\right\Vert ^{2}-\left\langle \mathbf{e\times\mathbf{k}},\mathbf{x}\right\rangle ^{2}=\left\langle \mathbf{k},\mathbf{x}\right\rangle ^{2}$.)

Let us evaluate $L_{P}$$(\varphi,\mathbf{x})$. Observe that a geodesic
$\gamma$ on $\mathbb{S}^{2}$ passing through a point $\mathbf{x}$
at direction of a unite vector $\mathbf{e}\in T_{\mathbf{x}}\mathbb{S}^{2}$
coincides with an orbit of one-parameter subgroup $\left\{ \mathrm{e}^{\Omega t}\right\} _{t\in\mathbb{R}}$
of the group $\mathrm{SO}(3)$. Such a geodesic is the rotation with
angular velocity $\mathbf{w}:=\mathbf{x}\times\mathbf{e}$, and $\Omega$
is a skew-symmetric operator such that $\Omega\mathbf{x}\equiv\mathbf{w}\times\mathbf{x}$.
In addition, for any $\mathbf{e}_{1}\in\in T_{\mathbf{x}_{0}}\mathbb{S}^{2}$
the mapping $t\mapsto\mathrm{e}^{\Omega t}\mathbf{e}_{1}$ is the
parallel translation along $\gamma$. Hence, for any $\mathbf{e},\mathbf{e}_{1}\in T_{\mathbf{x}}\mathbb{S}^{2}$,
we have
\begin{gather*}
\left|\left\langle \nabla P(\varphi,\mathbf{x})(\mathbf{e},\mathbf{e}_{1}),\mathbf{e}\right\rangle \right|=\left|\frac{\mathrm{d}}{\mathrm{d}t}\Bigl|_{t=0}\left\langle P\left(\varphi,\mathrm{e}^{\Omega t}\mathbf{x}\right)\mathrm{e}^{\Omega t}\mathbf{e}_{1},\mathrm{e}^{\Omega t}\mathbf{e}\right\rangle \right|=\\
\left|\frac{\mathrm{d}}{\mathrm{d}t}\Bigl|_{t=0}\left[\left\langle \left(\mathrm{e}^{t\Omega}\mathbf{e}_{1}\right)\times\mathbf{B}(\varphi),\mathrm{e}^{t\Omega}\mathbf{e}\right\rangle -\varkappa\left\langle \mathrm{e}^{t\Omega}\mathbf{e}_{1},\mathrm{e}^{t\Omega}\mathbf{e}\right\rangle \right]\right|=\\
\left|\frac{\mathrm{d}}{\mathrm{d}t}\Bigl|_{t=0}\left\langle \mathbf{e}_{1}\times\left(\mathrm{e}^{-t\Omega}\mathbf{B}(\varphi)\right),\mathbf{e}\right\rangle \right|=\left|\left\langle \mathbf{e}_{1}\times\left(\mathbf{B}(\varphi)\times\mathbf{w}\right),\mathbf{e}\right\rangle \right|=\\
\left\Vert \mathbf{B}(\varphi)\right\Vert \left|\left\langle \mathbf{k},\mathbf{e}\right\rangle \right|\left|\left\langle \mathbf{x}\times\mathbf{e},\mathbf{e}_{1}\right\rangle \right|.
\end{gather*}
 Since the maximum is attained when $\mathbf{e}_{1}=\mathbf{x}\times\mathbf{e}$,
$\mathbf{e}=\mathbf{k}-\left\langle \mathbf{k},\mathbf{x}\right\rangle \mathbf{x}$,
we find
\begin{gather*}
L_{P}(\varphi,\mathbf{x})=\left\Vert \mathbf{B}(\varphi)\right\Vert \sqrt{1-\left\langle \mathbf{k},\mathbf{x}\right\rangle ^{2}}
\end{gather*}

Thus conditions~(\ref{eq:4lambda-ge}) and (\ref{eq:labdaf-g-sigma}),
respectively, take the form
\begin{gather*}
\frac{a}{(1+2\rho a+a^{2})^{3/2}}>\frac{E\rho}{\sqrt{1-\rho^{2}}}+\frac{\left(\beta^{2}\rho^{2}+\varkappa^{2}\right)\rho}{2},
\end{gather*}
 $\mbox{\ensuremath{}}$

\begin{gather*}
\frac{a}{\left(1+2as+a^{2}\right)^{3/2}}-z_{\ast}\rho\sqrt{\left(1-s^{2}\right)}\left(2\sqrt{(\beta^{2}s^{2}+\varkappa^{2})}+\beta\right)-\\
\frac{s\rho^{2}(\beta^{2}s^{2}+\varkappa^{2})}{4}>0\quad\forall s\in[\rho,1].
\end{gather*}
Obviously, these inequalities are fulfilled once $\rho$ is small
enough to satisfy the condition
\begin{gather*}
\frac{a}{(1+a)^{3}}>\max\left\{ \rho z_{\ast}\left[\frac{2\beta^{2}+\varkappa^{2}}{\beta}+\frac{\rho(\beta^{2}+\varkappa^{2})}{4}\right],\frac{\rho E}{\sqrt{1-\rho^{2}}}+\frac{\rho\left(\beta^{2}\rho^{2}+\varkappa^{2}\right)}{2}\right\} .
\end{gather*}
Observe that on account of Remark~\ref{rem: Th-B-sol-f-P} we have
\begin{gather*}
\rho z_{\ast}\le\rho z_{+}+\sqrt{\frac{\rho z_{+}}{q}}(1-\rho^{2})^{1/4}(1+lz_{+})\sqrt{\left(E+\frac{1}{1+a^{2}}\right)}.
\end{gather*}
The above inequalities allow us to establish an upper bound for magnitude
of perturbation which does not destroy the quasiperiodic solution
obtained in Theorem~\ref{thm:charged-particle}.

It should be noted that in~\cite{BarGer10} the authors study trajectories
of autonomous system governing the motion of classical particles accelerated
by a potential and a magnetic field on a non-complete Riemannian manifold.

\section*{Final remarks}

Since Lyapunov proposed his direct method, the analysis of nonlinear
systems by means of auxiliary functions whose level sets are transversal
to vector fields of systems' right hand sides was successfully carried
out by many authors (see, e.g.,~\cite{KKMP94,KraPer84,MawWar02,OZVK13}
). The success in constructing such functions for concrete systems
depends on the art of researcher. In the case where the system~(\ref{eq:qp-sys-f})
is a Lagrangian one, i.e. $f(\varphi,x)=-\nabla\Pi(\varphi,x)$, it
is naturally to seek the auxiliary function $U(\cdot)$ using the
averaged function $\intop_{\mathbb{T}^{k}}\Pi(\varphi,x)\mathrm{d}\varphi$,
as was proposed in~\cite{ParRus12}. We will devote a separate paper
to further applications of the results obtained.

\end{document}